\documentclass{amsart}

\usepackage{amsmath}
\usepackage{amsopn}
\usepackage[british]{babel}
\usepackage{amssymb}
\usepackage{graphicx}
\usepackage{epstopdf}
\usepackage{color}
\usepackage{tikz}
\usepackage{diagbox}
\usepackage{pgfplots}
\usepackage{enumitem}
\usepackage{hyperref}
\usepackage{fp-eval}
\usepackage{esint}
\usepackage{subfigure}
\usepackage{caption}
\usepackage{mathabx}
\usepackage{booktabs}
\usepackage{todonotes}
\usepackage{subdepth}
\usepackage{mathtools}
\usepackage{bm}
\usepackage{cleveref}

\definecolor{darkblue}{rgb}{0,0,.7}
\hypersetup{colorlinks=true,linkcolor=darkblue,citecolor=darkblue}

\newlist{alphenum}{enumerate}{1}
\setlist[alphenum]{fullwidth,label={(\alph*)}}

\allowdisplaybreaks[0]
\theoremstyle{plain}
\newtheorem{theorem}{Theorem}
\newtheorem{lemma}[theorem]{Lemma}

\newtheorem{definition}[theorem]{Definition}
\theoremstyle{remark}
\newtheorem{remark}[theorem]{Remark}

\newcommand{\vecb}[1]{\boldsymbol{#1}}
\newcommand{\transformed}[1]{\widehat{#1}}

\renewcommand{\div}{\operatorname{div}}

\newcommand{\dom}{\transformed{\Omega}}
\newcommand{\axidom}{\Omega}

\newcommand{\f}{\bm{f}}
\newcommand{\g}{\bm{g}}
\newcommand{\bpsi}{\bm{\psi}}
\renewcommand{\u}{\bm{u}}
\renewcommand{\v}{\bm{v}}
\newcommand{\z}{\bm{z}}
\newcommand{\w}{\bm{w}}
\renewcommand{\bf}{\bm{f}}

\newcommand{\V}{\bm{V}}

\newcommand{\ds}{~\!\mathrm{d}s}

\newcommand{\drdz}{~\!\mathrm{d}r\mathrm{d}z}
\newcommand{\drda}{~\!\mathrm{d}r\mathrm{d}\alpha}
\newcommand{\dxtrans}{~\!\mathrm{d}\transformed{\bm{x}}}

\newcommand{\GammaR}{\Gamma_{\! \text{rot}}}
\newcommand{\TR}{\mathcal{T}_{\text{rot}}}

\newcommand{\Cerr}{\mathcal{R}_{\Pi}}

\newcommand{\BR}{BR }

\usepackage{tikz}
\usetikzlibrary{cd,patterns,calc,fit,backgrounds}
\usepackage{pgfplots,pgfplotstable}
\usepgfplotslibrary{colorbrewer,groupplots}
\usepackage{caption}

\title[Pressure-robustness for the axisymmetric Stokes problem]{Pressure-robustness for the axisymmetric Stokes problem by velocity reconstruction
}

\author{P.L.~Lederer}
\thanks{Department of Mathematics, University of Hamburg, Bundesstra{\ss}e 55, D-20146 Hamburg, Germany
  (\href{mailto:philip.lederer@uni-hamburg.de}{philip.lederer@uni-hamburg.de}).}
\author{C.~Lehrenfeld}
\thanks{Institute for Numerical and Applied Mathematics, University of Göttingen, Lotzestraße 16-18, D-37083 Göttingen, Germany
  (\href{mailto:lehrenfeld@math.uni-goettingen.de}{lehrenfeld@math.uni-goettingen.de}, \href{mailto:t.beeck@math.uni-goettingen.de}{t.beeck@math.uni-goettingen.de}).}
\author{C.~Merdon}
\thanks{Weierstrass Institute for Applied Analysis and Stochastics (WIAS),  Anton-Wilhelm-Amo-Straße 39, D-10117 Berlin, Germany
(\href{mailto:christian.merdon@wias-berlin.de}{christian.merdon@wias-berlin.de}).}
\author{T.~van Beeck}


\ifpdf
\hypersetup{
  pdftitle={Pressure-robustness for the axisymmetric Stokes problem by velocity reconstruction},
  pdfauthor={P.L.~Lederer, C.~Lehrenfeld, C.~Merdon, and T.~van Beeck},
}
\fi

\begin{document}

\begin{abstract}
This paper studies pressure-robust discretizations of the axisymmetric Stokes problem. After the transformation to cylindrical coordinates, the \emph{radially weighted} discrete velocity describing the mass flux has to be divergence-free in the classical sense to guarantee pressure-robustness. Classical divergence-free finite element methods from the Cartesian setting do not satisfy this condition in general and are thus not pressure-robust in the axisymmetric setting, even though they remain inf-sup stable.

We explore an approach that restores pressure-robustness via reconstruction operators for a low-order Bernardi--Raugel discretization. Applying standard interpolation operators from the Cartesian setting to radially weighted test functions works in principle, but the resulting operator lacks the properties needed to derive optimal consistency error estimates.
Therefore, we introduce a reconstruction operator mapping into a finite element space spanned by Raviart--Thomas functions modified to vanish on the rotation axis. This vanishing-on-axis property is the key to obtaining optimal consistency error estimates. Numerical examples confirm the theoretical results and include cases where the modified reconstruction operator yields significantly better results.
\end{abstract}

\maketitle

\section{Introduction}
Pressure-robust discretizations of the incompressible (Navier--)Stokes equations have become a major point of focus in recent years \cite{JLMNR:2017}. For example, such discretizations can be based on finite elements providing exactly divergence-free velocities, such as the Scott--Vogelius element \cite{SV85,Zhang:2011b} on suitable meshes or $H(\mathrm{div})$-conforming (hybrid) discontinuous Galerkin methods \cite{MR2304270, LS16, MR4122492, MR3833698}. An alternative approach is the use of a velocity reconstruction operator \cite{Linke2014, LM:2016, LLMS2017, JLMNR:2017} applied to test functions.
All of these schemes work in the two- and three-dimensional Cartesian setting. 

When the domain, data, and solution are rotationally invariant, the original three-dimensional problem can be reduced to a two-dimensional axisymmetric formulation, which can be solved at a significantly lower computational cost. 
Unfortunately, without further modifications, classical pressure-robust schemes are no longer pressure-robust under the associated change of variables. 
To be precise, let $\axidom$ denote the two-dimensional meridional $(r,z)$-domain and let $\dom$ be the three-dimensional domain obtained by revolving $\axidom$ about the rotation axis. 
We consider the three-dimensional Stokes problem: find $\u : \dom \subset \mathbb{R}^3 \to \mathbb{R}^3$ such that 
\begin{alignat*}{2}
  - \nu \Delta \u + \nabla p &= \bf &&\quad \text{ in } \dom, \\
  \div \u &= 0 &&\quad \text{ in } \dom, \\
  \u &= 0 &&\quad \text{ on } \partial \dom.
\end{alignat*}
Transforming to cylindrical coordinates $(r,\theta,z)$ and assuming that the problem is axisymmetric, that is, independent of $\theta$, we obtain a mixed problem of the form: find $\u \in \V$ and $p \in Q$ --- with $\V$ and $Q$ being appropriate function spaces for velocity and pressure on $\Omega$, respectively --- such that 
\begin{align*}
\nu a(\u,\v) + b(p,\v) & = (\bf, r \v)_{L^2} &\hspace*{-1.5cm}&  \text{for all } \v \in \V,\\
         b(q,\u) & = 0 &\hspace*{-1.5cm}& \text{for all } q \in Q,
\end{align*}
with the standard $L^2$-inner product $(\cdot, \cdot)_{L^2}$ (that is, without a weighting factor) over $\axidom$ and the bilinear forms 
\begin{align*}
a(\u,\v) := (r \nabla_{(r,z)} \u, \nabla_{(r,z)} \v)_{L^2} + (r^{-1} u_r, v_r)_{L^2}, \quad 
b(q,\v) := -(\mathrm{div}_{(r,z)}(r \v), q)_{L^2}.
\end{align*}
Here, the operators $\nabla_{(r,z)}$ and $\mathrm{div}_{(r,z)}$ denote the standard gradient and divergence with respect to the $(r,z)$ coordinates and $\cdot_r$ refers to the radial component, see e.g., \cite{BDM99} for a proper derivation of the above formulation. The forms $a(\cdot,\cdot)$ and $b(\cdot,\cdot)$ differ from the usual Stokes forms in Cartesian coordinates.
At first glance, the change looks minor, but the divergence operator's structure changes significantly under the cylindrical-coordinate transform. For instance, because of the additional radial factor $r$ one would now require $\mathrm{div}_{(r,z)}(r\V_h)\subseteq Q_h$ to obtain an \emph{exactly} divergence-free discretization. Consequently, well-known exactly divergence-free velocity ansatz spaces, e.g., those provided by the Scott--Vogelius finite element, do not satisfy this property and therefore do not yield a pressure-robust method; we illustrate this in Figure~\ref{fig:PRFailure}. A related phenomenon was observed in \cite{QRV21} for an acoustic eigenvalue problem, where the axisymmetric transform produces spurious eigenmodes.

\begin{figure}[!htbp]
  \centering
  \captionsetup{singlelinecheck=off}
  \includegraphics[width=0.97\textwidth]{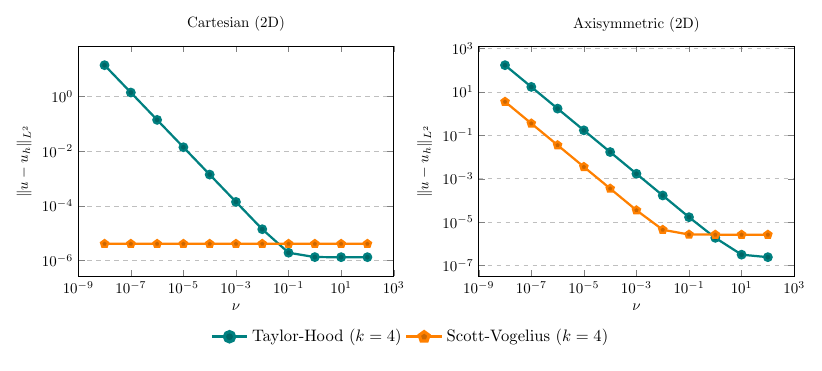}
  \caption{Illustrating pressure-robustness for the Taylor--Hood and Scott--Vogelius pairs in Cartesian (left) and axisymmetric (right) settings for right-hand sides $\bf=-\nu\Delta \u+\nabla p$. A pressure-robust method yields $\nu$-independent velocity errors, whereas non-robust methods show locking with respect to $\nu$. 
  Taylor--Hood is not pressure-robust and Scott--Vogelius is only pressure-robust in the Cartesian setting.
  In this example, $\bf$ is computed from
  \(\u = \operatorname{curl} (x^2 (x-1)^2 y^2 (y-1)^2)\) and \(p = x^5 + y^5 - \tfrac13\) in the Cartesian case, and as in
  Section~\ref{sec:example2} in the axisymmetric case.}
  \label{fig:PRFailure}
\end{figure}

Another difficulty stems from the $r^{-1}$ weight in the bilinear form $a(\cdot,\cdot)$: they require that the radial component $u_r$ vanishes at the rotation axis, which is natural for radially symmetric flows but complicates discretizations with $H(\mathrm{div})$-conforming ansatz spaces. 
To the best of the authors' knowledge, only a few exactly divergence-free axisymmetric discretizations on structured meshes are available \cite{BFO96, LZZZ}.

Furthermore, deriving inf-sup stability of a finite element pair from the three-dimensional setting is not straightforward; see \cite{LEE2011, LEE20123500} for the Taylor--Hood and \cite{MR1977001} for the P2-bubble finite element method.

\smallskip
This paper employs a classical low-order Bernardi--Raugel finite element discretization \cite{BR1985} and modifies it with the help of a reconstruction operator in the spirit of \cite{Linke2014, LM:2016, LLMS2017, JLMNR:2017} to achieve pressure-robustness. We present the Bernardi--Raugel pair here as the first axisymmetric pressure-robust scheme on unstructured grids and as a prototype for a broader class of pressure-robust axisymmetric discretizations.

All considered reconstruction operators $\Pi$ map
discretely divergence-free (radially weighted) test functions 
to exactly divergence-free $H(\mathrm{div})$-conforming
ones by utilizing the commutative property
$\mathrm{div}_{(r,z)} \Pi(r\v_h) = \pi_0 \mathrm{div}_{(r,z)} (r\v_h)$.
We discuss two approaches. The first one employs
classical $H(\mathrm{div})$-conforming standard interpolations, that are only guaranteed to work under higher regularity
assumptions on the problem data. These higher regularity assumptions are only realistic if the rotational axis is not part of the boundary of the meridional domain $\Omega$.
To fix this flaw, a second approach employs a modified
$H(\mathrm{div})$-conforming finite element space
in which the tangential component of
the reconstructed functions $\Pi(r \v_h)$ vanishes
at the rotation axis. 
This property allows $L^2_{-1}$ estimates of the reconstructions (where $-1$ indicates the scaling of the radial weight in the $L^2$ inner product) and is the key
for an improved and optimal consistency error estimate.
The theoretical results and the properties of the modified
reconstruction operators are confirmed in several  numerical experiments.
Moreover, the reconstruction operator applied to the
discrete velocity solution is divergence-free (in the sense above)
and can be used in coupled axisymmetric transport simulations to
guarantee maximum principles and mass conservation \cite{FMR26}.

Furthermore, the reconstruction space motivates $H(\mathrm{div})$-conforming (H)DG in the spirit of \cite{MR2304270, LS16, MR4122492, MR3833698}. A thorough investigation of this possibility is the topic
of a forthcoming paper.

\smallskip
The remaining parts of this paper are structured as follows.
Section~\ref{sec:modelproblem} studies the well-posedness
of the axisymmetric Stokes model problem and introduces the
necessary notation. Section~\ref{sec:pressure_robustness}
introduces the concept of pressure-robustness and derives
a general a priori error estimate. In Section~\ref{sec:classical_reconstructions}, we investigate the employment of traditional $H(\mathrm{div})$-conforming
standard interpolations. Section~\ref{sec:modified_reconstruction}
suggests a modification of the reconstruction operator
that allows for an improved consistency error estimate
and optimal a priori velocity error estimates without additional
regularity assumptions. Section~\ref{sec:numerical_examples}
presents several numerical examples to confirm the results.

\section{Model problem and preliminaries}
\label{sec:modelproblem}
This section fixes the notation for the axisymmetric setting and recalls the strong and weak formulation, as well as the Helmholtz decomposition.

\subsection{Axisymmetric geometry}
Let $\axidom \subset \mathbb{R}_{+} \times \mathbb{R}$, where $\mathbb{R}_{+}$ are the non-negative real numbers, be a bounded Lipschitz domain that generates a domain $\dom \subset \mathbb{R}^3$ through rotation around an axis. We denote by $(x,y,z)$ a set of Cartesian coordinates in $\mathbb{R}^3$ and by $(r,\theta,z)$, $r \ge 0$, $\theta \in [0,2\pi]$, the associated cylindrical coordinates such that $x = r \cos \theta$ and $y = r \sin \theta$. We assume that
\begin{equation*}
     \dom = \{ (r \cos \theta, r \sin \theta,z) : (r,z) \in \axidom, \theta \in [0,2\pi] \}
\end{equation*}
{and define the possibly empty set $\GammaR := \{ (r,z) \in \partial \axidom : r = 0 \}$}. We set $\Gamma \coloneqq \partial \axidom \setminus \GammaR$. The geometric setup is visualized in Figure~\ref{fig:GeomSetup}. For any vector, or vector-valued function $\v$, we denote by $(v_x,v_y,v_z)$ and $(v_{r},v_{\theta},v_{z})$ its components in the Cartesian and cylindrical coordinate system, respectively. In this work we only consider the axisymmetric case, i.e., $v_{\theta} = 0$ and all functions are independent of the angular coordinate $\theta$. For simplicity, we then only consider the two-dimensional vector $\v = (v_r,v_z)$ in the $(r,z)$-plane.

\begin{figure}[!htbp]
  \begin{center}
    \includegraphics[width=0.95\textwidth]{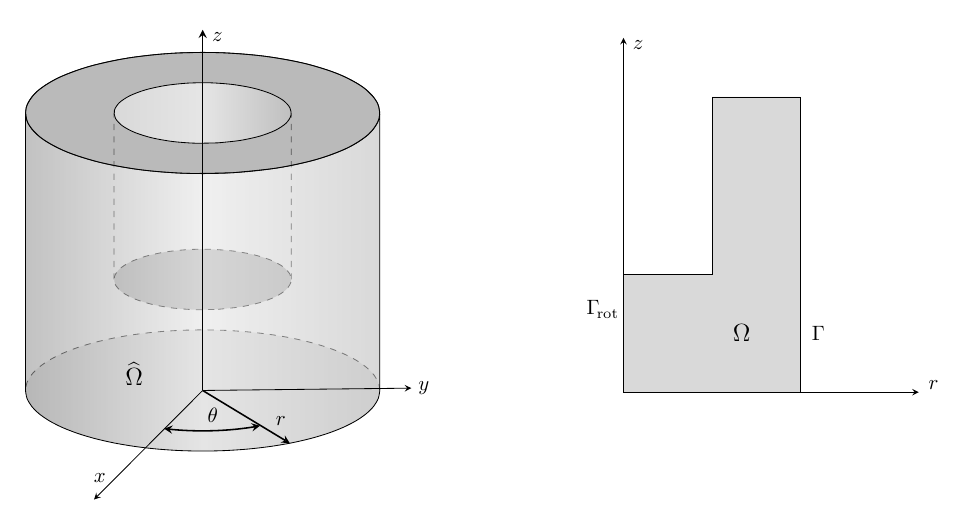}
  \end{center}
  \caption{Geometric setup: A two-dimensional meridional domain $\axidom$ generates a three-dimensional domain $\dom$ by rotation around the axis $\{r = 0\}$. The boundary of $\Omega$ may (left) or may not (right) contain the rotation axis.}
  \label{fig:GeomSetup}
\end{figure}

We mainly focus on the case where $\GammaR \neq \emptyset$, because the modifications of the reconstruction operators in Section~\ref{sec:modified_reconstruction} are only necessary in this case.

\subsection{Strong formulation and differential operators}
Considering homogeneous Dirichlet boundary conditions, the strong form of the axisymmetric Stokes problem seeks $\u = (u_r,u_z)$
and $p$ such that
\begin{alignat*}{2}
 - \nu \Delta_\text{axi} \u + \nabla p &= \bf &&\quad \text{in } \Omega,\\
                        \mathrm{div}_\text{axi} \u &= 0 &&\quad \text{in } \Omega,\\
                        \u &= 0 &&\quad \text{on } \Gamma,
\end{alignat*}
and involves the axisymmetric vector Laplacian and divergence operator
\begin{align*}
   \Delta_\text{axi} \u :=   \begin{pmatrix}
    \Delta_{(r,z)} u_r + \frac{1}{r} \partial_r u_r - r^{-2} u_r\\
    \Delta_{(r,z)} u_z + \frac{1}{r} \partial_r u_z
  \end{pmatrix} 
  \quad \text{and} \quad
  \mathrm{div}_\text{axi} :=  \mathrm{div}_{(r,z)}(\v) + r^{-1} v_r.
\end{align*}
Here and throughout
\begin{align*}
\Delta_{(r,z)} v := \partial^2_r v + \partial^2_z v, \quad
\nabla_{(r,z)} \v ~\! \vecb{e}_s:= \partial_s \v, s \in \{r,z\}, \text{ and }
\mathrm{div}_{(r,z)} \v := \partial_r v_r + \partial_z v_z,
\end{align*}
with the Jacobian matrix $\nabla_{(r,z)} \v$,
denote the classical differential operators known from the Cartesian setting but with respect to the cylindrical coordinates $(r,z)$. Note that non-homogeneous Dirichlet boundary conditions can be treated in a standard way via homogenization.

\subsection{Weak formulation and weighted Sobolev spaces}
\label{sec:weak_formulation}
The weak formulation employs weighted Sobolev spaces $L^2_s(\Omega)$
and $H^1_s(\Omega)$, see \cite{Kuf80,BDM99,CDH23}, that are Hilbert spaces with the inner products
\begin{align*}
   (p, q)_{L^2_s} := \int_\Omega r^s p q \drdz,
   \quad \text{and} \quad
   (p, q)_{H^1_s} := \int_\Omega r^s \left( p q + \nabla p \cdot \nabla q \right) \drdz,
\end{align*}
and the corresponding weighted norms
$\| q \|_{L^2_s}^2 := (q,q)_{L^2_s}$ and $\| q \|_{H^1_s}^2 := (q,q)_{H^1_s}$,
respectively. We omit the index $s$ in the notation if $s = 0$.
{We use boldface notation to denote vector-valued spaces and estimates with $\lesssim$ and $\gtrsim$ hide generic constants.}

The weak solution $(\u, p) \in \V \times Q$ of the axisymmetric Stokes solves the mixed problem
\begin{subequations}\label{eq:AxiStokesCont} 
  \begin{alignat}{2}
  \nu a(\u,\v) + b(p,\v) & = (\bf, r \v)_{L^2} && \quad \text{for all } \v \in \V,\\
           b(q,\u) & = 0 && \quad \text{for all } q \in Q,
  \end{alignat}
\end{subequations}
where the bilinear forms read
\begin{align*}
a(\u,\v) := (\nabla_{(r,z)} \u, \nabla_{(r,z)} \v)_{L^2_1} + (u_r, v_r)_{L^2_{-1}}
\quad \text{and} \quad
b(q,\v) := {-(\mathrm{div}_{(r,z)}(r \v), q)_{L^2}}.
\end{align*}
The space $\V := \lbrace v \in \bm{H}^1_1(\Omega) : v_r \in L^2_{-1}(\Omega),
{\v = 0 \text{ along } \Gamma} \rbrace$ is equipped with the inner product $a(\cdot,\cdot)$ and corresponding energy norm
\begin{align*}
  \| \v \|^2_{\V} := a(\v,\v) = \| \nabla \v \|^2_{L^2_1} + \| v_r \|^2_{L^2_{-1}}.
\end{align*}
Observe, that boundedness in the energy norm requires $v_r = 0$ along the rotation axis $\GammaR$ if it is part of the domain boundary. 
On the remaining boundary $\Gamma := \partial \Omega \setminus \GammaR$,
we assume homogeneous Dirichlet boundary conditions for simplicity.
The pressure space $Q := \lbrace q \in L^2_1(\Omega) : \int_\Omega r q \mathit{dx} = 0 \rbrace$
fixes the integral mean to ensure uniqueness. Well-posedness is inherited from
the three-dimensional setting via the transformations in Remark~\ref{rem:transformations_to_3d} below.
More details on the functional analysis and well-posedness for axisymmetric problems
can be found in \cite{BDM99}.

In the axisymmetric setting, it is the flux quantity $r \u$ that is conserved and hence divergence-free. 
With abuse of notation we denote the corresponding velocity $\u$ as divergence-free in the sense that
\begin{align*}
  \u \in \V_0
  := \lbrace \v \in \V : b(q, \v) = 0 ~~\text{for all}~~ q \in Q \rbrace.
\end{align*}
We denote the $a(\cdot,\cdot)$-orthogonal complement to $\V_0$ in $\V$ by $\V_\perp$, i.e., $\V = \V_0 \oplus_a \V_\perp$.
Motivated by the continuity estimate 
\begin{align*}
  \lvert b(q,\v) \rvert  = \lvert - (\mathrm{div}_{(r,z)}(r \v), q)_{L^2} \rvert 
  \le \| \mathrm{div}_{(r,z)}(r \v) \|_{{L^2_{-1}}} \| q \|_{L^2_1}
  \quad \textrm{for all } \v \in \V, q \in Q,
\end{align*}
and 
\begin{align}
  \| \mathrm{div}_{(r,z)}(r \v) \|_{{L^2_{-1}}}^2
  &= \int_\Omega r^{-1} \left( \mathrm{div}_{(r,z)}(r \v) \right)^2 \drdz
  = \int_\Omega r^{-1} \left( r \mathrm{div}_{(r,z)}(\v) + v_r \right)^2 \drdz, 
  \nonumber\\
  &
   \le 2 \int_\Omega r \left( \mathrm{div}_{(r,z)}(\v) \right)^2 \drdz + 2 \int_\Omega r^{-1} v_r^2 \drdz
  \lesssim \| \v \|^2_{\V}, \label{eq::div_continuity_estimate}
\end{align}
we further see that 
\begin{align} \label{eq::ru_in_hdivminusone}
  r \v \in H_{-1}(\mathrm{div},\Omega)
  := \left\lbrace \v \in \bm{L}^2_{-1}(\Omega) : \mathrm{div}_{(r,z)}(\v) \in L^2_{-1}(\Omega) \right\rbrace.
\end{align}
Here and in the following we omit the subscript $(r,z)$ of the divergence in the notation of the $H_{-1}(\mathrm{div},\Omega)$-space (and similar ones) for brevity. 

Sections~\ref{sec:classical_reconstructions} and \ref{sec:modified_reconstruction}
discuss pressure-robust finite element
discretizations that are based on reconstructions of the weighted (discrete) test functions $r \v_h$. The most sophisticated approach, as it allows for an error analysis with improved consistency error, is motivated by \eqref{eq::ru_in_hdivminusone} but is weakened in the sense that it only maps into
\begin{align} \label{eq::Piru_in_hdiv}
  \bm{L}^2_{-1}(\Omega) \cap H(\mathrm{div}, \Omega) 
  = \left\lbrace \v \in L^2_{-1}(\Omega) : \mathrm{div}_{(r,z)}(\v) \in L^2(\Omega) \right\rbrace.
\end{align}
 Thus, we only consider standard $H(\mathrm{div})$-conformity without the weighting factor $r$ in the norm-bound of the divergence. 

\begin{remark}[Transformation to the three-dimensional setting]\label{rem:transformations_to_3d}
  The bilinear forms are equivalent to the bilinear forms of the three-dimensional Stokes problem in Cartesian coordinates on the body of revolution $\widehat{\Omega}$ (up to a factor $2\pi$) by the transformation rules
\begin{align*}
  a(\u,\v)
  = \int_{\Omega} \left( \nabla \u : \nabla \v + r^{-2} u_r v_r \right) r \drdz
= \frac{1}{2\pi}\int_{\transformed{\Omega}} \nabla_{\!\transformed{\bm{x}}} \transformed{\u} : \nabla_{\!\transformed{\bm {x}}} \transformed{\v} \dxtrans,
\end{align*}
and
\begin{align*}
  b(q,\v)
  &= (\mathrm{div}_{(r,z)}(r \v), q)_{L^2}
  = \! \int_{\Omega} \! q \left( \mathrm{div}_{(r,z)}(\v) + r^{-1} v_r \right) r \drdz
  = \! \frac{1}{2\pi}\int_{\transformed{\Omega}} \! \transformed{q} \mathrm{div}_{\transformed{\bm x}} (\transformed{\v}) \dxtrans.
\end{align*}
\end{remark}

\subsection{Helmholtz projection}\label{sec:helmholtz_projection}
The Helmholtz decomposition splits any axisymmetric forcing into a gradient (pressure) part and a divergence-free remainder. Concretely, for any $\bm f\in L^2_1(\Omega)$ there exist
\begin{align}\label{eqn:helmholtz_decomposition}
  \bm f = \nabla \alpha_{\bm f} + \mathbb{P}\bm f,
\end{align}
with $\alpha_{\bm f}\in H^1_1(\Omega)/\mathbb{R}$ determined by
\begin{align*}
  (\nabla \alpha_{\bm f},\nabla \beta)_{L^2_1} = (\bm f,\nabla\beta)_{L^2_1}\qquad\text{for all }\beta\in H^1_1(\Omega)/\mathbb{R}.
\end{align*}
The remainder $\mathbb{P}\bm f \in \bm{L}^2_1(\Omega)$ is $L^2_1$-orthogonal to all gradients, i.e., $(\mathbb{P}\bm f,\nabla q)_{L^2_1}=0$ for all $q\in H^1_1(\Omega)$.
Furthermore, $\nabla \alpha_{\bm f}$ is $L^2_1$-orthogonal to all divergence-free functions $\v_0\in\V_0$, since
$$
(\nabla \alpha_{\bm f},\v_0)_{L^2_1} = (\nabla \alpha_{\bm f}, r\v_0)_{L^2} = -(\alpha_{\bf}, \mathrm{div}_{(r,z)}(r \v_0))_{L^2} = 0.
$$
We extend the Helmholtz projector to functionals $\g\in\V^\star$ by restriction to the divergence-free subspace $\V_0$: for $\v\in\V$ with $\v = \v_0 + \v_\perp$, $\v_0 \in \V_0$, $\v_\perp \in \V_\perp$, {define $\mathbb{P}\g \in \V_0^\star$ by}
\begin{align*}
  \langle \mathbb{P}\g,\v\rangle_{\V_0^\star\times\V} := \langle\g,\v_0\rangle_{\V^\star\times\V}.
\end{align*}
If the functional $\g\in\V^\star$ is represented by a field $\bm f\in L^2_1(\Omega)$ then $\mathbb{P}\g \in \V_0^\star$ coincides with $\mathbb{P}\bm f \in L^2_1(\Omega)$ in the sense
\begin{align*}
  \langle\g,\v_0\rangle_{\V^\star\times\V} \! = (\bf,\v_0)_{L^2_1} =  (\nabla \alpha_{\bf},\v_0)_{L^2_1} + (\mathbb{P}\bm f,\v_0)_{L^2_1} = (\mathbb{P}\bm f,\v_0)_{L^2_1} \quad \text{for all }\v_0\in\V_0,
\end{align*}
by orthogonality of $\nabla\alpha_{\bm f}$ to $\V_0$.
An important application is $\g=-\Delta_\text{axi}\u\in\V^\star$, defined via
\begin{align}\label{eq::def_lap_axi_vstar}
  \langle -\Delta_\text{axi}\u,\v\rangle_{\V^\star\times\V} := a(\u,\v)\qquad\text{for all }\v\in\V.
\end{align}
Testing this identity {for the solution $\u$ from \eqref{eq:AxiStokesCont}} with $\v_0\in\V_0$ gives
\begin{align*}
  \langle\mathbb{P}(-\Delta_\text{axi}\u),\v_0\rangle = a(\u,\v_0)=\frac{1}{\nu}(\bm f,\v_0)_{L^2_1}=\frac{1}{\nu}(\mathbb{P}\bm f,\v_0)_{L^2_1}.
\end{align*}
Hence, even if $-\Delta_\text{axi}\u$ need not lie in $\bm{L}^2_1(\Omega)$, its Helmholtz projection admits the $\bm{L}^2_1$-representation $\nu^{-1}\mathbb{P}\bm f$.

\section{Pressure-robustness for the axisymmetric Stokes problem
by reconstruction operators}
\label{sec:pressure_robustness}

In the Cartesian setting, pressure-robustness of a non-divergence-free method can be achieved by applying a reconstruction operator to the right-hand side, see \cite{Linke2014,LMT2016,LM:2016,LLMS2017}. The main idea is to replace the discretely divergence-free test functions with exactly divergence-free ones that are conforming at least in an $H(\mathrm{div})$ sense. In this section, we study the application of a generic reconstruction operator to radially-weighted test functions in the axisymmetric setting.

\subsection{A modified finite element method}
\label{sec:modified_fem}
Consider a regular triangulation $\mathcal{T}$ of $\Omega$ into triangles
with vertices $\mathcal{N}$ and edges $\mathcal{E}$,
and let $P_k(\mathcal{T})$ be the space of piecewise polynomials of order $k$.
For each $T \in \mathcal{T}$, let $h_T := \mathrm{diam}(T)$ denote its diameter and let $h_\mathcal{T} \in P_0(\mathcal{T})$ denote the piecewise constant mesh size function defined by $h_\mathcal{T}|_T := h_T$,
and as usual $h := \max_{T \in \mathcal{T}} h_T$.
\label{sec:modified_scheme}
We assume at least the regularity $\bf \in L^2_1(\Omega)$, corresponding to
$\widehat{\bf} \in L^2(\widehat{\Omega})$ in the three-dimensional transformed problem.
The modified axisymmetric Stokes problem seeks $(\u_h, p_h) \in \V_h \times Q_h$, such that
\begin{subequations}\label{eq:AxiStokesDiscr}
  \begin{alignat}{2}
  \nu a(\u_h,\v_h) + b(p_h,\v_h) & = (\bf, \Pi(r \v_h))_{L^2} && \quad \textrm{for all } \v_h \in \V_h, \\
          b(q_h,\u_h) & = 0 && \quad \textrm{for all } q_h \in Q_h,
  \end{alignat}
\end{subequations}
where $\Pi : \V \rightarrow \mathbf{W}_h$ is a suitable reconstruction operator
into a discrete $H(\mathrm{div}, \Omega)$-conforming space $\mathbf{W}_h$.
Note that with the choice $\Pi = \mathbb{I}$ (identity operator), we recover the classical finite element method without modifications. 

The well-posedness of the problem hinges on the inf-sup stability of the pair $\V_h \times Q_h$ and the
continuity of the modified right-hand side. For the latter, it suffices that the reconstruction operator provides the continuity estimate 
\begin{align} \label{eqn:continuity_reconstruction}
 \| \Pi(r \v_h) \|_{L^2_{-1}} \lesssim \| \v_h \|_{\V} \quad \text{for all } \v_h \in \V_h,
\end{align}
which then gives for $\f \in L^2_1(\Omega)$ 
\begin{align} \label{eqn:continuity_modified_rhs}
  (\f, \Pi(r \v_h))_{L^2} & \leq \| \f \|_{L^2_1} \| \Pi(r \v_h) \|_{L^2_{-1}} \le \| \f \|_{L^2_1} \| \v_h \|_{\V}. 
\end{align}

In this work, we consider the lowest-order Bernardi--Raugel finite element method
with the velocity and pressure ansatz spaces
\begin{align*}
  \V_h := \bm{P}_1(\mathcal{T}) \oplus \lbrace \lambda_i \lambda_j \bm{n}_E : E = \mathrm{conv} \lbrace N_i, N_j\rbrace \in \mathcal{E}\rbrace, \qquad
  Q_h := P_0(\mathcal{T}) \cap L^2_{1,0}(\Omega).
\end{align*}
Here, $\lambda_j$ is the barycentric coordinate for the vertex $N_j$
and $\bm{n}_E$ is a fixed unit normal vector of the edge $E \in \mathcal{E}$.
The space of discretely divergence-free functions is then given by
\begin{align*}
  \V_{h,0} := \left\lbrace \v_h \in \V_h : b(q_h, \v_h) = 0 \text{ for all } q_h \in Q_h \right\rbrace.
\end{align*}
The inf-sup stability of this pair
can be shown, similar to the proof in
\cite[Section 2.2.1]{LEE2011} for the Taylor--Hood finite element method, by using
Scott--Zhang type quasi-interpolators to design a Fortin interpolator.
Details can be found in \cite{FMR26}.

Motivated by \cite{LM:2016}, a straightforward choice is the standard interpolation into the $\textrm{RT}_0(\mathcal{T})$
and $\textrm{BDM}_1(\mathcal{T})$ finite element spaces defined in Section~\ref{sec:classical_reconstructions}.
As these choices have some drawbacks, some enhancement is discussed in Section~\ref{sec:modified_reconstruction}.
Before that, this section concludes with some abstract a priori error estimates.

\subsection{A priori error estimates}
The following abstract a priori error estimate includes consistency errors caused by the application or non-application of a reconstruction operator $\Pi$ in \eqref{eq:AxiStokesDiscr}, which we estimate further in the following sections. 

Apart from the Helmholtz decomposition of the right-hand side $\bf$ from \eqref{eqn:helmholtz_decomposition}, the analysis utilizes the discrete Stokes projector $\Pi_S: \V \to \V_{h,0}$ and the continuous Stokes lifting $\mathcal{L}_S: \V_h \to \V_0$ defined via 
\begin{alignat*}{2}
   a(\Pi_S \v, \v_h)&= a(\v , \v_h)   &&\quad \text{for all } \v_h \in \V_{h,0}, \ \v \in \V, \\
  a(\mathcal{L}_S\v_h, \v) &=a(\v_h , \v)     &&\quad \text{for all } \v \in \V_0, \ \v_h \in \V_h.
\end{alignat*}
Observe that there holds
\begin{align}\label{eqn:stokes_proj_lift_commute}
    a(\Pi_S \v, \w_h) = a(\v, \w_h) = a(\v, \mathcal{L}_S \w_h) &\quad \text{for all } \v \in \V_0, \w_h \in \V_{h,0}.
\end{align}

Under some common regularity assumption, the lifting allows for the following error estimate.
\begin{lemma} \label{lem::stokes_lifting_error_estimate}
If the Stokes problem inherits $\bm{H}^{1+s}_1(\Omega) \times H^s_1(\Omega)$ elliptic regularity
for some $s \in (0, 1]$ (from the equivalent three-dimensional problem)
such that $\nu \|\u\|_{\bm{H}^{1+s}_1} + \|p \|_{H^s_1} \lesssim \| \bf \|_{L^2_1}$,
there holds
    \begin{align}\label{eqn:stokes_lifting_error_estimate}
        \| \w_h - \mathcal{L}_S \w_h \|_{L^2_1} \lesssim \| h_\mathcal{T}^s \div(r \w_h) \|_{L^2_{-1}} \lesssim \| h_\mathcal{T}^s \w_h \|_{\V}
        \quad \text{for all } \w_h \in \V_{h,0}.
    \end{align}
\end{lemma}
\begin{proof}
  The proof follows the proof from \cite[Lemma 5.2]{LMN20}. Let $(\bpsi, \lambda)$ denote the
  solution of the auxiliary axisymmetric Stokes problem
  \begin{alignat*}{2}
    a(\bpsi, \z) + b(\lambda, \z) & = (\w_h - \mathcal{L}_S \w_h, \z)_{L^2_1} &&\quad \text{for all } \z \in \V,\\
    b(q, \bpsi) & = 0 &&\quad \text{for all } q \in Q.
  \end{alignat*}
  Testing with $\z = \w_h - \mathcal{L}_S \w_h$, using the definition of
  $\mathcal{L}_S$ and inserting $\pi_{Q_h} \lambda$, i.e., the $L^2_1$-best approximation of $\lambda$ in $Q_h$,
  yields
  \begin{align*}
    \| \w_h - \mathcal{L}_S \w_h \|_{L^2_1}^2
    & = a(\bpsi, \w_h - \mathcal{L}_S \w_h) 
    + b(\lambda, \w_h) 
    - b(\lambda, \mathcal{L}_S \w_h) = 0 + b(\lambda, \w_h) + 0\\
    & = b(\lambda - \pi_{Q_h} \lambda,  \w_h) \qquad {( \text{as } b(q_h, \w_h)=0 \text{ for all } q_h \in Q_h)}\\
    & \lesssim \| h_\mathcal{T}^{-s} (\lambda - \pi_{Q_h} \lambda) \|_{L^2_1} \| h_\mathcal{T}^s \div(r \w_h) \|_{L^2_{-1}}
    \lesssim {\| \lambda \|_{H^1_s}}  \| h_\mathcal{T}^s \div (r \w_h) \|_{L^2_{-1}}.
  \end{align*}
  {Using elliptic regularity
  $\| \lambda \|_{H^1_s} \lesssim \| \w_h - \mathcal{L}_S \w_h \|_{L^2_1} $
  and following} the estimates in \eqref{eq::div_continuity_estimate} concludes the proof.
\end{proof}

We can show that the discretization error is bounded by the best-approximation error in $\V_{h,0}$, the $L^2_1$-norm of the Helmholtz projector of $\Delta_\text{axi} \u$, and a \emph{consistency error} defined as follows: 
\begin{align}\label{eq:ConsistencyError}
  \Cerr \coloneqq \| \mathbb{P} (\Delta_\text{axi} \u) \circ (\mathbb{I} - \Pi) \|_{\V_{h,0}^\star} + \nu^{-1} \| \nabla \alpha_{\bf} \circ \Pi \|_{\V_{h,0}^\star}. 
\end{align}
Here, the dual norms are defined as 
\begin{align}\label{eqn:consistency_error}
  \| \mathbb{P} (\Delta_\text{axi} \u) \circ (\mathbb{I} - \Pi) \|_{\V_{h,0}^\star}
  & := \sup_{\w_h \in \V_{h,0}} \frac{(\mathbb{P}(\Delta_\text{axi} \u), r \w_h - \Pi (r \w_h))_{L^2}}{\| \w_h \|_{\V}},\\
  \| \nabla \alpha_{\bf} \circ \Pi \|_{\V_{h,0}^\star}
  & := \sup_{\w_h \in \V_{h,0}} \frac{(\alpha_{\bf}, \mathrm{div}(\Pi (r \w_h)))_{L^2}}{\| \w_h \|_{\V}},
\end{align}
where $\mathbb{P} (\Delta_\text{axi} \u)$ is identified with its representation $-\nu^{-1} \mathbb{P} \bf \in L^2_1(\Omega)$ as discussed at the end of Section~\ref{sec:helmholtz_projection}, and $\nabla \alpha_{\bf}$
is the remainder in the Helmholtz decomposition \eqref{eqn:helmholtz_decomposition} of $\bf$. 

\begin{theorem}[A priori error estimate]\label{thm:apriori_error_estimate}
    In addition to the assumptions of Lemma~\ref{lem::stokes_lifting_error_estimate}, we assume that $\bf \in L^2_1(\Omega)$ and that $\Pi$ allows for the continuity estimate \eqref{eqn:continuity_reconstruction}.
    Then, the error between the exact velocity $\u \in \V$ solving \eqref{eq:AxiStokesCont} and the discrete velocity $\u_h \in \V_h$ solving \eqref{eq:AxiStokesDiscr} is bounded by
    \begin{align*}
        \| \u - \u_h \|_{\V} \lesssim \inf_{\v_h \in \V_{h,0}} \| \u - \v_h \|_{\V} + \|h_\mathcal{T}^s  \mathbb{P} (\Delta_\text{axi} \u) \|_{L^2_1} + \Cerr.
    \end{align*}
\end{theorem}
\begin{proof}
  We first note that $\u \in \V_0$ and $\u_h \in \V_{h,0}$ and by definition of the discrete Stokes projector  
  we have $\w_h := \Pi_S \u - \u_h \in \V_{h,0}$ and
   $a( \Pi_S \u - \u, \Pi_S \u - \u_h) = 0$, i.e., $\Pi_S \u - \u \perp_a \Pi_S \u - \u_h$.
  Hence, the error can be split into a best-approximation error and a consistency error via the Pythagorean identity
  \begin{align*}
  \| \u - \u_h \|^2_{\V}
  & = \| \u - \Pi_S \u \|^2_{\V} + \| \Pi_S \u - \u_h \|^2_{\V},
  \end{align*}
  where $\| \u - \Pi_S \u \|^2_{\V} = \inf_{\v_h \in \V_{h,0}} \| \u - \v_h \|^2_{\V}$. It hence remains to estimate the consistency error term $\| \Pi_S \u - \u_h \|^2_{\V}$.
    As $b(q_h, \w_h) = 0$ for all $q_h \in Q_h$, we further obtain 
    \begin{align*} 
    a(\u_h, \w_h)
      & = \nu^{-1} (\bf, \Pi(r \w_h))_{L^2} 
      = \nu^{-1} (\mathbb{P} \bf + \nabla \alpha_{\bf}, \Pi(r \w_h))_{L^2} \\
      & = (\mathbb{P} (-\Delta_\text{axi} \u), \Pi(r \w_h))_{L^2}
      - \nu^{-1} (\alpha_{\bf}, \mathrm{div}(\Pi(r \w_h)))_{L^2}.
    \end{align*}
    Due to \eqref{eqn:continuity_reconstruction} and estimates similar to \eqref{eqn:continuity_modified_rhs}, all terms are well-defined. Adding and subtracting $(\mathbb{P} (-\Delta_\text{axi} \u), r \w_h)_{L^2}$ yields
    \begin{multline}
  a(\u_h, \w_h)
  = (\mathbb{P} (-\Delta_\text{axi} \u), \Pi(r \w_h) - r \w_h)_{L^2} + (\mathbb{P}(- \Delta_\text{axi} \u), \w_h)_{L^2_1}\\
  - \nu^{-1} (\alpha_{\bf}, \mathrm{div}(\Pi(r \w_h)))_{L^2}.\label{eq::cons_err_one}
    \end{multline}
    Next note, that we have 
    \begin{align} \label{eq::cons_err_two}
    (\mathbb{P} (- \Delta_\text{axi} \u), \w_h)_{L^2_1}
    &= (\mathbb{P} (- \Delta_\text{axi} \u), \w_h - \mathcal{L}_S \w_h)_{L^2_1} + (\mathbb{P} (- \Delta_\text{axi} \u), \mathcal{L}_S \w_h)_{L^2_1} \nonumber\\
    &= (\mathbb{P} (- \Delta_\text{axi} \u), \w_h - \mathcal{L}_S \w_h)_{L^2_1} + a(\Pi_S \u, \w_h),
    \end{align}
  where the last step follows from
  \begin{align*}
    (\mathbb{P}(- \Delta_\text{axi} \u), \mathcal{L}_S \w_h)_{L^2_1}
    \stackrel{\eqref{eq::def_lap_axi_vstar}}{=} a(\u,  \mathcal{L}_S \w_h) \stackrel{\eqref{eqn:stokes_proj_lift_commute}}{=} a(\Pi_S \u, \w_h).
  \end{align*}
  This shows that 
  \begin{align*}
  \| \w_h \|_{\V}^2 &= a(\w_h, \w_h) = a(\u_h, \w_h) - a(\Pi_S \u, \w_h) \\
  & \stackrel{\hspace*{-2em}\eqref{eq::cons_err_one}\hspace*{-2em}}{=} ~
  (\mathbb{P} (- \Delta_\text{axi} \u), \Pi(r \w_h) - r \w_h)_{L^2} + (\mathbb{P} (- \Delta_\text{axi} \u), \w_h)_{L^2_1} - a(\Pi_S \u, \w_h)\\
  & \hspace{7cm} - \nu^{-1} (\alpha_{\bf}, \mathrm{div}(\Pi(r \w_h)))_{L^2}\\
  & \stackrel{\hspace*{-2em}\eqref{eq::cons_err_two}\hspace*{-2em}}{=}~
  (\mathbb{P} (- \Delta_\text{axi} \u), \Pi(r \w_h) - r \w_h)_{L^2} + (\mathbb{P} (- \Delta_\text{axi} \u), \w_h - \mathcal{L}_S \w_h)_{L^2_1}\\
  & \hspace{7cm} - \nu^{-1} (\alpha_{\bf}, \mathrm{div}(\Pi(r \w_h)))_{L^2}\\
  & 
  \stackrel{\hspace*{-2em}\eqref{eqn:stokes_lifting_error_estimate}\hspace*{-2em}}{\le}~
   \left( \| \mathbb{P} (\Delta_\text{axi} \u) \! \circ \! (\mathbb{I} - \Pi) \|_{\V_{h,0}^\star} \! \! + \|h_\mathcal{T}^s \mathbb{P} (\Delta_\text{axi} \u) \|_{L^2_1} + \nu^{-1} \| \nabla \alpha_{\bf} \circ \Pi \|_{\V_{h,0}^\star} \! \right)\! \| \w_h\|_{V},
  \end{align*}
  which concludes the proof.
\end{proof}

\begin{remark}[Lack of pressure-robustness for $\Pi = \mathbb{I}$]
  While the consistency error $\| \mathbb{P} (\Delta_\text{axi} \u) \circ (\mathbb{I} - \Pi) \|_{\V^\star_{h,0}}$ in \eqref{eq:ConsistencyError} vanishes for $\Pi = \mathbb{I}$, the remaining term in $\Cerr$ is only bounded by the best-approximation error in the pressure space
  \begin{align}\label{eq:NablaAlpha}
    \| \nabla \alpha_{\bf} \circ \Pi \|_{\V_{h,0}^\star} = \| \nabla \alpha_{\bf}  \|_{\V_{h,0}^\star} = \inf_{q_h \in Q_h} \| \nabla (\alpha_{\bf} - q_h) \|_{\V_{h,0}^\star} & \lesssim \inf_{q_h \in Q_h} \| \alpha_{\bf} - q_h \|_{L^2_1}.
  \end{align}
  Since it appears with the factor $\nu^{-1}$ in \eqref{eq:ConsistencyError}, it causes a locking effect in the limit $\nu \rightarrow 0$ if the pressure cannot be approximated well enough.
  This characterizes the typical non-pressure-robust behavior of classical finite element methods.
\end{remark}

\begin{remark}[Condition for pressure-robustness]
  To avoid the locking effect outlined in the previous remark, the term \eqref{eq:NablaAlpha} needs to vanish. To ensure this, a potential reconstruction operator needs to guarantee the pointwise divergence-free condition
  \begin{align}\label{eqn:div_free_condition_reconstruction}
      \div(\Pi(r \v_h)) = 0 \quad \text{for all } \v_h \in \V_{h,0}.
  \end{align}
\end{remark}

\section{Classical $H(\mathrm{div}$)-conforming reconstruction operators} \label{sec:classical_reconstructions}
If the pressure space $Q_h$ consists of piecewise discontinuous polynomials
of order $0$, the standard interpolation operators $\Pi_\mathrm{BDM1}$  or $\Pi_\mathrm{RT0}$ into the Brezzi--Douglas--Marini (BDM) finite element space
$$\mathrm{BDM}_1(\mathcal{T}) := \bm P_1(\mathcal{T}) \cap H(\mathrm{div}, \Omega),$$
or the space of
Raviart--Thomas (RT) finite element functions
\begin{align*}
   \mathrm{RT}_0(\mathcal{T})
   := \Big\lbrace 
   \v_h(r,z)|_T = \begin{pmatrix}a\\b\end{pmatrix} + c\begin{pmatrix}r\\z\end{pmatrix} : a,b,c \in \mathbb{R} \text{ for all } T \in \mathcal{T} \Big\rbrace \cap H(\mathrm{div}, \Omega),
\end{align*}
are potential candidates for the reconstruction operator $\Pi$. Recall that $H(\mathrm{div},\Omega)$-conformity for polynomials requires continuous normal traces over all faces $\mathcal{F}$
of the triangulation.
For these standard interpolations, the following commutative property holds: 
\begin{align}\label{eqn:reconst_commute_prop}
  \mathrm{div}_{(r,z)} \Pi(r \v) = \pi_0 \mathrm{div}_{(r,z)}(r \v) \in Q_h \quad \text{for all } \v \in L^2(\Omega) \text{ and } r\v \in H(\mathrm{div},\Omega),
\end{align}
where $\pi_0$ is the $L^2$-orthogonal projection onto $P_0(\mathcal{T})$.
This directly ensures \eqref{eqn:div_free_condition_reconstruction}.
Moreover, first-order approximation properties hold in the sense
\begin{align}\label{eqn:reconst_approx_prop}
  \| h_\mathcal{T}^{-1} (1-\Pi) \v \|_{L^2} \lesssim \|\nabla \v\|_{L^2} \quad \text{for all } \v \in \bm{H}^1(\Omega).
\end{align}
Applied to a weighted test function {$\v \in \V$}, this yields
\begin{align}\label{eqn:reconst_approx_prop_weighted}
  \| h_\mathcal{T}^{-1} (1-\Pi) (r\v) \|_{L^2} & \leq \| \nabla (r\v)\|_{L^2}
                            \leq \|v_r\|_{L^2} + \|v_z\|_{L^2} + \| \nabla \v \|_{L^2_1}
                            \lesssim \| \v \|_{\V}
\end{align}
where we used the argument from Remark~\ref{rem:embedding_mr82} below and the chain rule
\begin{align*}
  \nabla (r \v) = r \nabla \v + \begin{pmatrix}
  v_r & 0 \\
  v_z & 0 \\
  \end{pmatrix}.
\end{align*}
\begin{remark}\label{rem:embedding_mr82}
   According to \cite[Remarque 4.1]{MR82}, $H^1_1(\Omega)$ is
   continuously embedded in $L^2_{-1 + \epsilon}(\Omega)$ for any $\epsilon > 0$, that is, there exists a constant $C_\epsilon$ such that
   \begin{align}
    \Vert v \Vert_{L^2_{-1 + \epsilon}(\Omega)} \leq C_\epsilon \Vert v \Vert_{H^1_1(\Omega)} \quad \text{for all } v \in H^1_1(\Omega).
   \end{align}
   For $\epsilon = 1$, this shows
   $v_z \in L^2(\Omega)$ for all $\v \in \V$ and hence $r \bm{v} \in H^1(\Omega)$. In particular, the last inequality in \eqref{eqn:reconst_approx_prop_weighted} follows from an application of the Friedrich's inequality.
\end{remark}

\begin{remark}[Regularity issue]\label{rem:regularity_issue}
  It is not guaranteed that the continuity estimates \eqref{eqn:continuity_reconstruction} and \eqref{eqn:continuity_modified_rhs} hold for the standard reconstructions $\Pi_\text{RT0}$ and $\Pi_\text{BDM1}$.
  However, if the right-hand side has the (higher) regularity $\bf \in L^2(\Omega)$, then we still get
  \begin{align*}
    (\bf, \Pi(r \v_h))_{L^2} & \leq \| \bf \|_{L^2} \| \Pi(r \v_h) \|_{L^2}
    \lesssim \| \bf \|_{L^2} \| r\v_h\|_{H^1}
    \lesssim \| \bf \|_{L^2} \| \v_h\|_{\V},
  \end{align*}
  {using Cauchy-Schwarz, a Friedrich's inequality and \eqref{eqn:reconst_approx_prop_weighted}}. Unfortunately, if only $\bf \in L^2_1(\Omega) \setminus L^2(\Omega)$, well-posedness is not guaranteed in general. 
  Indeed, a counter example for \eqref{eqn:continuity_reconstruction} relies on the fact that the reconstruction $\Pi(r \v_h)$ does not need to vanish at the rotation axis $\GammaR$. For the function $\w_h := (r, - 2z) \in \V_0$ the $\textrm{RT}_0(\mathcal{T})$ reconstruction of $r \w_h$ must be divergence-free and constant and hence $\Pi_\text{RT0}(r \w_h) \in P_0(\mathcal{T})$. At the rotation axis that constant will not be zero in general, which yields
  $\| \Pi_\text{RT0}(r \w_h) \|_{L^2_{-1}} = \infty$. If $\GammaR = \emptyset$, the $\Vert \cdot \Vert_{L^2_{-1}(\Omega)}$ and the $\Vert \cdot \Vert_{L^2(\Omega)}$ norm are equivalent and thus, assuming higher regularity is  mainly problematic in the cases where $\GammaR \neq \emptyset$.
\end{remark}

Besides the regularity issue above, the consistency error in Theorem~\ref{thm:apriori_error_estimate} can be estimated for the standard reconstruction operators under slightly stronger regularity assumptions on the data.

\begin{theorem}\label{thm:recons_consistency_error}
If there holds $\mathbb{P} (\Delta_\text{axi} \u) \in L^2(\Omega)$ (which is guaranteed by $\bf \in L^2(\Omega))$,
the consistency error \eqref{eqn:consistency_error}
for $\Pi \in \lbrace \Pi_\mathrm{RT0}, \Pi_\mathrm{BDM1} \rbrace$ can be estimated by
\begin{align*}
    \| \mathbb{P} (\Delta_\text{axi} \u) \circ (\mathbb{I} - \Pi) \|_{\V_{h,0}^\star}
    \lesssim \| h_\mathcal{T} \mathbb{P} (\Delta_\text{axi} \u) \|_{L^2} 
    \end{align*}
\end{theorem}
\begin{proof}
    This follows by applying the Cauchy-Schwarz inequality and \eqref{eqn:reconst_approx_prop_weighted}, i.e.,
    \begin{align*}
(\mathbb{P} (\Delta_\text{axi} \u), r \w_h - \Pi (r \w_h))_{L^2}
  & \leq \| h_\mathcal{T} \mathbb{P} (\Delta_\text{axi} \u) \|_{L^2} \| h_\mathcal{T}^{-1} (1-\Pi) (r \w_h) \|_{L^2}\\
  & \lesssim \| h_\mathcal{T} \mathbb{P} (\Delta_\text{axi} \u) \|_{L^2} \|\w_h \|_{\V}. 
    \end{align*}
\end{proof}

\begin{remark}
To solve the regularity issue, we could try to look for a reconstruction operator acting on $\v$ instead of $r\v$. A possible choice is an operator from \cite{copeland_mixed_2008} for an
axisymmetric Maxwell problem, which (adapted to our setting) reads as
\begin{align*}
   \widetilde{\Pi}_\mathrm{RT0}(\v) := \sum_{E \in \mathcal{E}} \psi_E^{\mathrm{RT}0} \frac{\int_E \v \cdot \bm{n} r \mathit{ds}}{\int_E r \mathit{ds}},
\end{align*}
where $\psi_E^{\mathrm{RT}0}$ are the $\mathrm{RT}_0$ basis functions for the edge $E \in \mathcal{E}$.
While $r \widetilde{\Pi}_\mathrm{RT0}(\v) \neq \Pi_\mathrm{RT0}(r \v)$,
the operators satisfy
\begin{align*}
  \int_T \mathrm{div}_{(r,z)}(\widetilde{\Pi}_\mathrm{RT0}(\v) r) \mathit{dx}
  &= \int_{\partial T} \widetilde{\Pi}_\mathrm{RT0}(\v) \cdot \bm{n} r \mathit{ds}
   = \int_{\partial T} \v \cdot \bm{n} r \mathit{ds}\\
  & = \int_{T} \mathrm{div}_{(r,z)}(\Pi_\mathrm{RT0}(r \v)) \mathit{dx}
  = \int_{T} \mathrm{div}_{(r,z)}(r \v) \mathit{dx},
\end{align*}
or, in other words,
\begin{align*}
   \mathrm{div}_{(r,z)}(\Pi_\mathrm{RT0}(r \v))
   = \pi_0(\mathrm{div}_{(r,z)}(r \v))
   = \pi_0(\mathrm{div}_{(r,z)}(\widetilde{\Pi}_\mathrm{RT0}(\v) r)).
\end{align*}
Unfortunately, it holds $\mathrm{div}_{(r,z)}(\widetilde{\Pi}_\mathrm{RT0}(\v) r) \notin Q_h$
in general, so using $\widetilde{\Pi}_\mathrm{RT0}(\v) r$ in the right-hand side
would probably solve the regularity issue (see \cite[Lemma 5.3]{copeland_mixed_2008}),
but will not result in a pressure-robust discretization requiring $\mathrm{div}_{(r,z)}(\Pi(\v r)) \in Q_h$.
\end{remark}

\section{A modified reconstruction operator}
\label{sec:modified_reconstruction}
This section presents a design for a modified reconstruction operator
that ensures pointwise homogeneous boundary conditions along $\GammaR$, which we assume to be non-empty for this section.
This condition is the key to get a reconstruction operator that maps $r \u \in H_{-1}(\mathrm{div},\Omega)$
into an $L^2_{-1}(\Omega)$-conforming discrete subspace of $H(\mathrm{div},\Omega)$, see \eqref{eq::Piru_in_hdiv}. The slightly more involved
construction allows lowering the regularity assumptions on $\mathbb{P}(\Delta_\text{axi} \u)$
in the consistency error estimate. In comparison to \eqref{eqn:reconst_approx_prop_weighted},
the improved error estimate
\begin{align}\label{eqn:interpolation_error_modified}
  \| h_\mathcal{T}^{-1} (1-\Pi) (r\v) \|_{L^2_{-1}} & \lesssim \| \v \|_{\V},
\end{align}
is shown in Theorem~\ref{thm:modified_interpolation_estimate} below.
The novel design requires the modification of the
ansatz functions on triangles intersecting the rotation axis $\GammaR$.

\subsection{Modified $L^2_{-1}(\Omega) \cap H(\mathrm{div}, \Omega)$-conforming interpolation}
For any triangle $T \in \mathcal{T}$, let $\mathcal{E}(T)$ be the set of all edges of $T$. We define the set 
\begin{align}\label{eq:ER}
  \mathcal{E}_R \coloneqq \lbrace E \in \mathcal{E} : E \cap \GammaR \neq \varnothing, E \not\subseteq \GammaR \rbrace,
\end{align}
collecting all interior edges that have one vertex on the rotation axis. Further, let $\TR \coloneqq \{ T \in \mathcal{T} : T \cap \GammaR \neq \emptyset\}$ be the set of triangles adjacent to the rotation axis. 
We distinguish two types of such triangles (see Figure~\ref{fig:type_1_2_triangles}): a \emph{type~1} triangle has exactly one vertex on $\GammaR$, while a \emph{type~2} triangle has exactly two vertices on $\GammaR$. Both types have two edges belonging to $\mathcal{E}_R$ and require special care in some
parts of the analysis for weighted Sobolev norms.

We are interested in a reconstruction operator into the space of piecewise linear functions that vanish at the rotation axis, and which is $H(\mathrm{div})$-conforming in the sense
of normal-continuity. In this case, considering a type~2 triangle,
such a polynomial must have the form \(r \bm{P}_0(T)\). Unfortunately, this is not a classical lowest order Raviart--Thomas function (if it is non-zero), but motivates to extend the space by certain polynomial functions. 
In the following, we design a suitable reconstruction operator based on modified Raviart--Thomas basis functions, which can be understood as a constrained $\textrm{BDM}_1$ space.

\begin{figure}[htbp]
  \centering
  \includegraphics[width=0.8\textwidth]{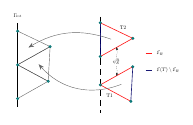}
  \caption{Illustration of type~2 ($\text{T}2$) and type~1 ($\text{T}1$) triangles adjacent to the rotation axis $\GammaR$. Both types of triangles have two edges belonging to $\mathcal{E}_R$ on which we define a modified basis function $\psi_E^{\mathrm{R}}$ in \eqref{eqn:modified_RT0_function}.}
  \label{fig:type_1_2_triangles}
\end{figure}

\begin{figure}
  \centering
  \begin{minipage}[c]{0.88\textwidth}
    \begin{tabular}{@{}c@{\hspace{2pt}}c@{\hspace{2pt}}c@{\hspace{2pt}}c@{}}
      \footnotesize$\psi_{E_1}^{\mathrm{RT}0}$ &
      \footnotesize$\psi_{E_2}^{\mathrm{RT}0}$ & 
      \footnotesize$\psi_{E_3}^{\mathrm{RT}0}$ &
      \footnotesize$\psi_{E_4}^{\mathrm{RT}0}$ \\[2pt]
      \includegraphics[width=0.235\linewidth]{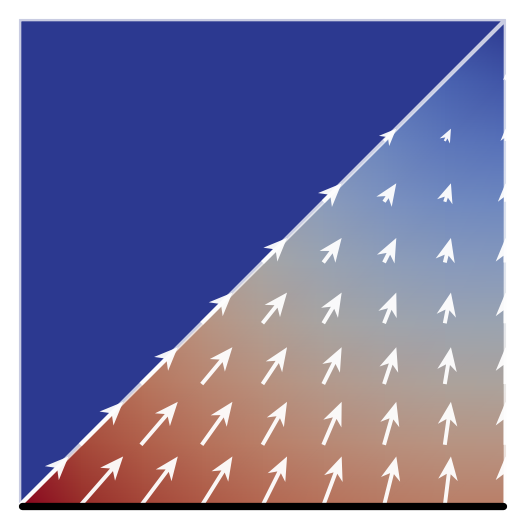} &
      \includegraphics[width=0.235\linewidth]{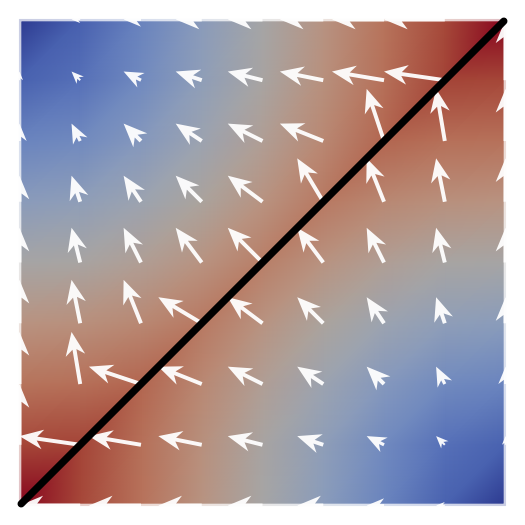} &
      \includegraphics[width=0.235\linewidth]{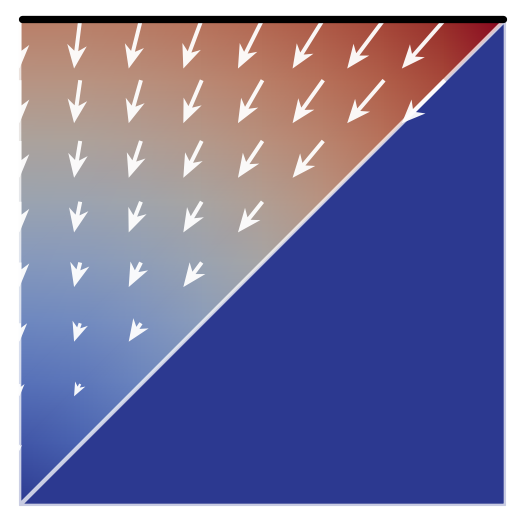} &
      \includegraphics[width=0.235\linewidth]{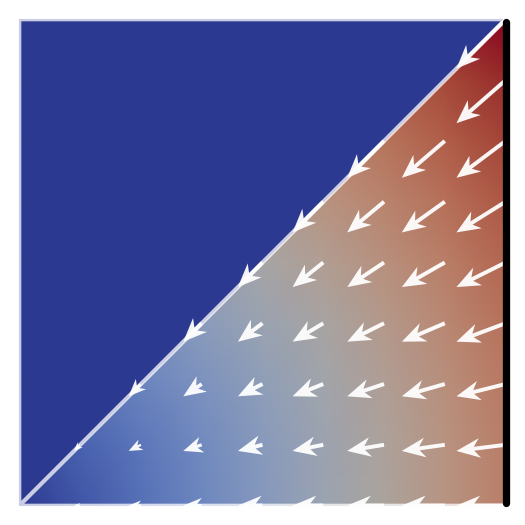} \\
      \footnotesize$\psi_{E_1}^{\mathrm{R}}$&
      \footnotesize$\psi_{E_2}^{\mathrm{R}}$&
      \footnotesize$\psi_{E_3}^{\mathrm{R}}$&
      \footnotesize$\psi_{E_4}^{\mathrm{RT0}}$\\[2pt]
      \includegraphics[width=0.235\linewidth]{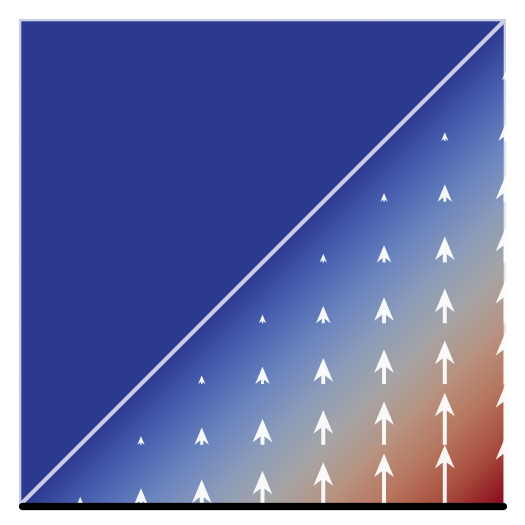} &
      \includegraphics[width=0.235\linewidth]{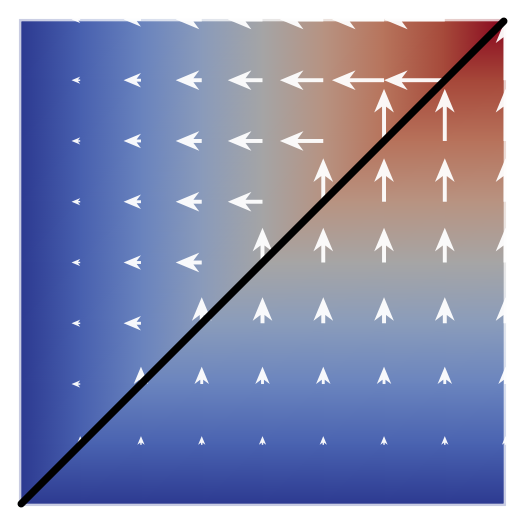} &
      \includegraphics[width=0.235\linewidth]{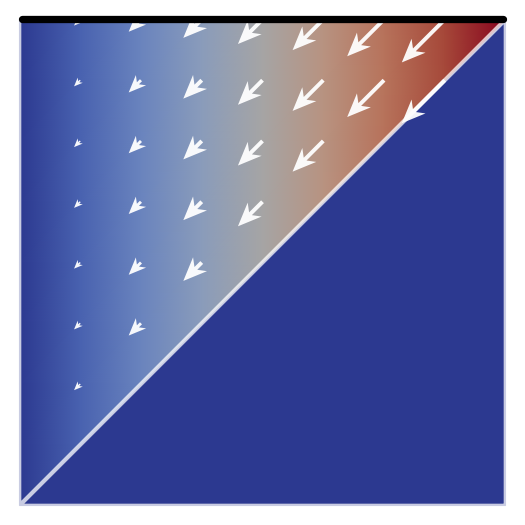} &
      \includegraphics[width=0.235\linewidth]{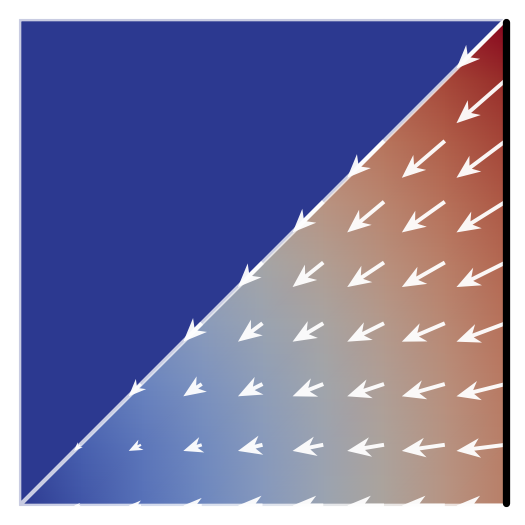}
    \end{tabular}
  \end{minipage}\hspace{0.01\textwidth}%
  \begin{minipage}[c]{0.07\textwidth}
    \vspace*{0.36cm}
    \includegraphics[width=\textwidth]{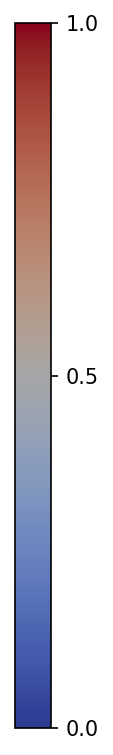}
  \end{minipage}
  \caption{\label{fig:modified_RT0_functions}
    Top row: Standard Raviart--Thomas basis functions on a simple two element mesh ($\GammaR$ corresponds to the left edge).
    Bottom row: Modified Raviart--Thomas basis functions. Edges corresponding to the shape functions are highlighted.
    The color encodes the normalized magnitude.
  }
\end{figure}

Consider an edge $E = \mathrm{conv} \lbrace N_i, N_j \rbrace \in \mathcal{E}_R$
with vertex $N_j = (r_j, z_j) \in \GammaR$ on the rotation axis
(recall that both type 1 and type 2 triangles have two of them).
For this edge $E$, we define a modified $\mathrm{RT}_{0}$ function by
\begin{align}\label{eqn:modified_RT0_function}
  \psi_{E}^\mathrm{R} := 2 \mathrm{curl} (\lambda_j) \lambda_i,
\end{align}
where $\lambda_i$ and $\lambda_j$ are the standard barycentric coordinates associated to the vertices $N_i$ and $N_j$, respectively. 
It can be obtained by a linear combination of the standard
$\mathrm{RT}_{0}$ and $\mathrm{BDM}_{1}$ basis functions which (up to scaling factors)
usually have the form
\begin{align}\label{eq::std_basisfunctions}
  \psi_E^{\mathrm{RT}0}  := \mathrm{curl} (\lambda_i) \lambda_j - \mathrm{curl} (\lambda_j) \lambda_i
  \quad \text{and} \quad
  \psi_E^{\mathrm{BDM}1} := \mathrm{curl} (\lambda_i \lambda_j).
\end{align}
Figure~\ref{fig:modified_RT0_functions} displays the four
classical and the four modified Raviart--Thomas basis functions on a triangulation
with two triangles, one type 1 and one type 2 triangle.
Observe how all modified basis functions vanish at the rotation axis on the left.
The following Lemma collects some useful properties of the modified basis functions.

\begin{lemma}[Properties of modified basis functions]\label{lem:props_mod_basis}
  The function \eqref{eqn:modified_RT0_function}
  for an edge $E = \mathrm{conv} \lbrace N_i, N_j \rbrace \in \mathcal{E}_R$
  with $N_j \in \GammaR$ and $N_i \notin \GammaR$
  has the following properties:
  \begin{itemize}
     \item[(i)] $\int_E \psi^{\mathrm{R}}_E \cdot \bm{n} \ds = \pm 1$ (sign depends on orientation of $\bm{n}$), 
     \item[(ii)] $\int_{\tilde E} \psi^{\mathrm{R}}_E \cdot \bm{n} \ds = 0$ for all other edges $\tilde E \in \mathcal{E}$ with $\tilde E \neq E$,
     \item[(iii)] $\psi_E^{\mathrm{R}}(N_j) = 0$,
     \item[(iv)] $\| \psi^{\mathrm{R}}_E \|_{L^2_{-1}(T)} \lesssim h_T^{-1/2}$.
  \end{itemize}
  Moreover, in a type 1 triangle on the edge $E \in \mathcal{E}(T) \setminus \mathcal{E}_R$ away from
  the rotation axis, it also holds
  \begin{itemize}
     \item[(v)] $\| \psi^{\mathrm{RT}0}_E \|_{L^2_{-1}(T)} \lesssim h_T^{-1/2}$,
     \item[(vi)] $\| \psi^{\mathrm{BDM}1}_E \|_{L^2_{-1}(T)} \lesssim h_T^{-1/2}$.
  \end{itemize}
  {On triangles $T \in \mathcal{T} \setminus \TR$, it holds
  $\| \psi^{\mathrm{RT}0}_E \|_{L^2_{-1}(T)} \! \! \approx \! (\min_{x \in T} r(x))^{-1/2}  \! \approx \! \| \psi^{\mathrm{BDM}1}_E \|_{L^2_{-1}(T)}.$}
\end{lemma}

\begin{proof}
  The first three properties follow from straightforward calculations. For property (iv), let $z_j$ be the $z$-coordinate of the vertex $N_j \in \GammaR$. Then the value of $\lambda_i$ at a point $(r,z)$ can be expressed as $\lambda_i = \nabla \lambda_i \cdot (r, z - z_j)$, because $\lambda_i$ is zero at the vertex $N_j$. 
This gives
  \begin{align*}
    \| \psi^{\mathrm{R}}_E \|^2_{L^2_{-1}(T)}
    & = 4 \lvert \mathrm{curl}(\lambda_j) \rvert^2 \int_T \frac{\lambda_i^2}{r} \drdz\\
    & = 4 \lvert \mathrm{curl}(\lambda_j) \rvert^2 \int_T \frac{\left( \nabla \lambda_i \cdot (r, z - z_j)\right)^2}{r} \drdz\\
    & \leq 4 \lvert \mathrm{curl}(\lambda_j) \rvert^2 \lvert \nabla(\lambda_j) \rvert^2 \int_T \frac{r^2 + (z - z_j)^2}{r} \drdz.
  \end{align*}
  The last integral can be transformed with the
 transformation $(r,z-z_j) = \Phi(r,\alpha) = (r, r \sin(\alpha))$ and $\mathrm{det}(D \Phi) = r \cos(\alpha)$, where $\alpha \in [\alpha_\text{min}, \alpha_\text{max}]$ is the angle between the $r$-axis and the line connecting the point $(r,z) \in T$ with the vertex $N_j$. Then we obtain
 \begin{align*}
 \int_T \frac{r^2 + (z - z_j)^2}{r} \drdz
 & = \int_{\alpha_\text{min}}^{\alpha_\text{max}} \int_0^{R(\alpha)} r^2(1 + \sin(\alpha))^2 \cos(\alpha) \drda\\
 & \leq \max_{\alpha \in (\alpha_\text{min}, \alpha_\text{max})} \lvert 1 + \sin(\alpha) \rvert^2 h_T^3.
 \end{align*}
 The proof of (v) and (vi) is very similar to the proof of (iv).
 This concludes the proof.
\end{proof}

\begin{definition}[Modified interpolation operators] \label{def::modifiedinterpolation}
We recall the definition of the set $\mathcal{E}_R$ in \eqref{eq:ER} for which
the modified basis functions shall be used.
Then we define the space
\begin{align*}
  \mathrm{RT}^{\mathrm{axi}}_0(\mathcal{T}) :=
  \mathrm{span}
  \Big[
 \Big \lbrace \psi^{\mathrm{R}}_E : E \in \mathcal{E}_R \Big\rbrace
  \cup \Big\lbrace \psi_E^{\mathrm{RT}0} : E \in \mathcal{E} \setminus \mathcal{E}_R \Big\rbrace
  \Big],
\end{align*}
where linear independence follows from property (i) and (ii) of Lemma~\ref{lem:props_mod_basis}. Further, we define the space 
\begin{align*} 
  \mathrm{BDM}^{\mathrm{axi}}_1(\mathcal{T})
  &:=
  \mathrm{RT}^{\mathrm{axi}}_0(\mathcal{T})
  +
  \mathrm{span}
  \left\lbrace \psi_E^{\mathrm{BDM}1} : E \in \mathcal{E} \setminus \mathcal{E}_R \right\rbrace.
\end{align*}
The modified interpolations $\Pi_{\mathrm{RT}0}^\mathrm{axi} : \V \rightarrow \mathrm{RT}^{\mathrm{axi}}_0(\mathcal{T})$ and $\Pi_{\mathrm{BDM}1}^\mathrm{axi} : \V \rightarrow \mathrm{BDM}^{\mathrm{axi}}_1(\mathcal{T})$ of some function $\w \in \V$ are uniquely characterized by
\begin{align} \label{eq::modified_rt_dofs} 
  \int_E (\Pi_{\mathrm{RT}0}^\mathrm{axi}\w - \w) \cdot \bm{n} \ds
  &= 0 \quad \text{for all } E \in \mathcal{E},
\end{align}
and 
\begin{subequations}\label{eq::modified_bdm_dofs}
\begin{align}
  \int_E (\Pi_{\mathrm{BDM}1}^\mathrm{axi}\w - \w) \cdot \bm{n} \ds 
  &= 0 \quad \text{for all } E \in \mathcal{E}_R,\\
  \int_E (\Pi_{\mathrm{BDM}1}^\mathrm{axi}\w - \w) \cdot \bm{n} \,r_h \ds
  &= 0 \quad \text{for all } E \in \mathcal{E}\setminus \mathcal{E}_R, r_h \in P_1(E).
\end{align}
\end{subequations}
\end{definition}

Note, that we will also use local element-wise versions of the above interpolation operators. In this case the integrals in \eqref{eq::modified_rt_dofs} and \eqref{eq::modified_bdm_dofs} are only taken over the edges of the respective element, and we interpolate into the space $\mathrm{RT}^{\mathrm{axi}}_0(T)$ and $\mathrm{BDM}^{\mathrm{axi}}_1(T)$, respectively. To simplify the notation, we will not introduce a new symbol for the local interpolation operators but rather use the restriction operator, e.g., in Lemma~\ref{lem::consistency} in the next section. As usual, if the interpolated function is smooth enough, e.g., an element of $\V$, the restriction of the global interpolation operator to an element $T$ coincides with the local interpolation operator on that element.

\begin{lemma} \label{lem:modified_BDM1_space}
  There holds 
\begin{align*}
  \mathrm{BDM}^{\mathrm{axi}}_1(\mathcal{T})
   = \Big\lbrace \v_h \in \mathrm{BDM}_1(\mathcal{T}) :
  \v_h = \bm{0} \text{ on } \GammaR \Big\rbrace.
\end{align*}
\end{lemma}
\begin{proof}
  First note, that the standard $\mathrm{BDM}_1(\mathcal{T})$ basis function \eqref{eq::std_basisfunctions}, associated to edges $E \in \mathcal{E} \setminus \mathcal{E}_R$ vanish on the rotation axis since $\lambda_i=\lambda_j=0$  on $\GammaR$. The same holds for all functions in $ \mathrm{RT}^{\mathrm{axi}}_0(\mathcal{T})$ due to Lemma \ref{lem:props_mod_basis}, property (iii).
  
  We will work with a dimension argument. For this let $\v_h \in \mathrm{BDM}_1(\mathcal{T})$ with $\v_h = \bm{0}$ on $\GammaR$. On a type 2 triangle, the constraint $\v_h = \bm{0}$ on $\GammaR$ reduces the dimension of the local linear polynomial space from 6 to 2, which is the number of modified basis functions on that triangle (which are linearly independent by Lemma \ref{lem:props_mod_basis}). 
  On a type 1 triangle, the constraint $\v_h = \bm{0}$ on $\GammaR$ reduces the dimension of the local linear polynomial space from 6 to 4. 
  Again, the two modified basis functions and the two standard $\mathrm{BDM}_1$ basis functions associated to the edge $E \in \mathcal{E}(T) \setminus \mathcal{E}_R$ are linearly independent.   
  Thus, in both cases the dimensions of the local spaces coincide. Since the modified basis functions are also normal continuous (being a linear combination of standard $\mathrm{BDM}_1$ basis functions), we conclude the proof.
\end{proof}

\subsection{Improved consistency error estimate}
{This section shows an improved interpolation estimate for the modified
reconstruction operators $\Pi_{\mathrm{RT}0}^\mathrm{axi}$ and $\Pi_{\mathrm{BDM}1}^\mathrm{axi}$.}

In preparation, we show the following lemma.

\begin{lemma} \label{lem::consistency}
  Let $T \in \TR$. For all $\mathbf{C} \in \mathbf{P}_0(T)$, we have that 
  \begin{align*}
      \Pi_{\mathrm{BDM}1}^\mathrm{axi}(r \mathbf{C})|_T = r \mathbf{C}
      \quad \text{for all } \mathbf{C} \in \mathbf{P}_0(T).
   \end{align*}
   Further, if $T$ is of type 2, then $\Pi_{\mathrm{RT}0}^\mathrm{axi}(r \mathbf{C})|_T = r \mathbf{C}$ for all $\mathbf{C} \in \mathbf{P}_0(T)$. If $T$ is of type 1, then 
   \begin{align}\label{eqn:constant_reconstruction_error_estimate}
      \| r \mathbf{C} - \Pi_{\mathrm{RT}0}^\mathrm{axi}(r \mathbf{C}) \|_{L^2_{-1}(T)}
      \lesssim \lvert E \rvert \lvert r_2 - r_3 \rvert \lvert \mathbf{C} \rvert
      \| \psi_E^{\mathrm{BDM}1} \|_{L^2_{-1}(T)},
   \end{align}
   where $E \in \mathcal{E}(T) \setminus \mathcal{E}_R$ is the edge away from the rotation axis and $r_2 - r_3$ is the difference of the $r$-coordinates of the two vertices of $E$. 
\end{lemma}

\begin{proof}
  First consider the local $\Pi_{\mathrm{BDM}1}^\mathrm{axi}$ interpolator. In both cases, type 1 and type 2 triangles, the linear function $r \mathbf{C}$ vanishes on $\GammaR$. Thus, with the same dimension argument as in the proof of Lemma~\ref{lem:modified_BDM1_space}, it can be written as a linear combination of the modified basis functions on type 2 triangles and as a linear combination of the modified basis functions and the standard $\mathrm{BDM}_1$ basis function associated to the edge away from the rotation axis on type 1 triangles. This shows that $\Pi_{\mathrm{BDM}1}^\mathrm{axi}(r \mathbf{C})|_T = r \mathbf{C}$ for all $\mathbf{C} \in \mathbf{P}_0(T)$. 

The same argument also holds for $\Pi_{\mathrm{RT}0}^\mathrm{axi}$ on type 2 triangles. On type 1 triangles, the argument does not work since we are missing the additional (standard) $\mathrm{BDM}_1$ basis function on the edge away from the rotation axis. However, the above arguments in combination with the degrees of freedoms given by \eqref{eq::modified_rt_dofs} and \eqref{eq::modified_bdm_dofs} show that the difference can be written as
\begin{multline} \label{eqn:constant_reconstruction_error}
\left\lvert r \mathbf{C} - \Pi_{\mathrm{RT}0}^\mathrm{axi}(r \mathbf{C})|_T \right\rvert
= 
\left\lvert \Pi_{\mathrm{BDM}1}^\mathrm{axi}(r \mathbf{C})|_T - \Pi_{\mathrm{RT}0}^\mathrm{axi}(r \mathbf{C})|_T \right\rvert\\
      = \left \lvert \int_{E} r (\lambda(s) - 1/2) \mathbf{C} \cdot \mathbf{n} \, ds \right\rvert \, \lvert \psi_E^{\mathrm{BDM}1} \rvert,
   \end{multline} 
   where $\lambda$ is one of the barycentric coordinates on $E$. Evaluating the integral with a Simpson rule yields the claim.
\end{proof}

Additionally, we require the following Poincar\'e and trace estimates 
on triangles adjacent to the rotation axis.
\begin{lemma}
\label{lem:weighted_poincare}
  Let $T \in \TR$. For all $\mathbf{v} \in H^1_1(T)$ with $ \int_T rv dx = 0$, it holds that
   \begin{align*}
        \| \v \|_{L^2_1(T)} \lesssim h_T \| \nabla \v \|_{L^2_1(T)}.
     \end{align*}
\end{lemma}
\begin{proof}
   This follows from transformation to the 3D setting and the
   standard Poincar\'e inequality.
\end{proof}

\begin{lemma}
  Let $T \in \TR$. For any edge $E \in \mathcal{E}(T)$, it holds that
  \begin{align}\label{eqn:coefficient_estimate}
      \int_E r \v \cdot \bm{n} \ds
      \lesssim h_T^{3/2} \| \nabla \v \|_{L^2_1(T)},
  \end{align}
   for all $\v \in \bm{H}^1_1(T)$ with $\int_T r \v \mathit{dx} = 0$.
\end{lemma}

\begin{proof}
  From \cite[Lem.~A.1]{copeland_mixed_2008}, we have the trace inequalities
  \begin{align*}
        \| r \v \|^2_{L^2(E)} & \lesssim \| \v \|^2_{L^2_1(T)} + h_T^2 \| \nabla \v \|_{L^2_1(T)}^2 && \text{on type 1 triangles},\\
        \| \v \|^2_{L_1^2(E)} & \lesssim h_T^{-1} \| \v \|^2_{L^2_1(T)} + h_T \| \nabla \v \|_{L^2_1(T)}^2 && \text{on type 2 triangles}.
  \end{align*} 
  The claim follows from a suitable Cauchy inequality, the trace inequalities above and standard scaling arguments.
\end{proof}

\begin{theorem}[$L^2_{-1}$ interpolation error estimate for $\Pi_{\mathrm{RT}0}^\mathrm{axi}$]
  \label{thm:modified_interpolation_estimate}
  Let $T \in \TR$. For any function $\v \in \bm{H}^1_1(T)$ with $v_r \in L^2_{-1}(T)$, it holds that 
  \begin{align*}
    \| r \v - \Pi_{\mathrm{RT}0}^\mathrm{axi}(r \v) \|_{L^2_{-1}(T)}
    \lesssim \begin{cases} 
    h_T \| \nabla \v \|_{L^2_1(T)} & \text{if } T \text{ is of type 2,}\\
    h_T \| \nabla \v \|_{L^2_1(T)} + h_T^{-\epsilon} \lvert r_3 - r_2 \rvert \| \v \|_{L^2_{-1 + 2\epsilon}(T)} & \text{if } T \text{ is of type 1},
   \end{cases}
  \end{align*}
  where $\epsilon > 0$ is arbitrary (having in mind that $\v \in L^2_{-1 + \epsilon}(T)$, see Remark~\ref{rem:embedding_mr82}).
  Globally, there holds 
  \begin{align*}
    \| h_\mathcal{T}^{-1 + \epsilon} \left(r \v - \Pi_{\mathrm{RT}0}^\mathrm{axi}(r \v)\right) \|_{L^2_{-1}(\Omega)}
    \lesssim \| \v \|_{\V}
    \quad \text{for all } \v \in \V.
  \end{align*}

Analogous results hold for $\Pi_{\mathrm{BDM}1}^\mathrm{axi}$ with $\epsilon = 0$, where the
local type 2 estimate also holds for type 1 triangles.
\end{theorem}
\begin{proof} We start with the estimate for $\Pi_{\mathrm{RT}0}^\mathrm{axi}$, the $\Pi_{\mathrm{BDM}1}^\mathrm{axi}$ case is discussed at the end of the proof.
  Let $T \in \TR$ be arbitrary. After adding and subtracting $\overline{\v} := \int_T r \v \mathit{dx} / \int_T r \mathit{dx} \in \mathbf{P}_0(T)$
    one obtains with two triangle inequalities
    \begin{multline} \label{eq::interpolationestimate_split}
       \| r \v - \Pi_{\mathrm{RT}0}^\mathrm{axi}(r \v) \|_{L^2_{-1}(T)}\\
       \leq \| \v - \bar \v \|_{L^2_{1}(T)} + \| \Pi_{\mathrm{RT}0}^\mathrm{axi}(r (\v - \bar \v)) \|_{L^2_{-1}(T)}
       + \| r \bar \v - \Pi_{\mathrm{RT}0}^\mathrm{axi}(r \bar \v) \|_{L^2_{-1}(T)}.
    \end{multline}
    The first term can be estimated with the Poincar\'e inequality
    from Lemma~\ref{lem:weighted_poincare} and the second by
    a combination of \eqref{eqn:coefficient_estimate} and properties (iv) - (vi) from Lemma~\ref{lem:props_mod_basis}. On type 2 triangles, the last term vanishes by Lemma~\ref{lem::consistency}. On type 1 triangles, the consistency estimate \eqref{eqn:constant_reconstruction_error_estimate} of Lemma~\ref{lem::consistency} yields
     \begin{align*} 
       \| r \bar \v - \Pi_{\mathrm{RT}0}^\mathrm{axi}(r \bar \v) \|_{L^2_{-1}(T)}
       \lesssim h_T^{1/2} \lvert r_3 - r_2 \rvert \lvert \bar \v \rvert,
     \end{align*}
    recalling $\vert E \vert \simeq h_T$ and $\| \psi_E^{\mathrm{BDM}1} \|_{L^2_{-1}(T)} \lesssim h_T^{-1/2}$ by Lemma~\ref{lem:props_mod_basis}. 
    Estimating the constant by
    \begin{align*}
       \lvert \bar \v \rvert \lesssim h_T^{-3/2} \| \v \|_{L^2_1(T)}
       \lesssim h_T^{-1/2 - \epsilon} \| \v \|_{L^2_{-1 + 2\epsilon}(T)},
    \end{align*}
    for any $\epsilon > 0$ gives
     \begin{align}\label{eq:RT0:local_interpolationEstimate}
       \| r \bar \v - \Pi_{\mathrm{RT}0}^\mathrm{axi}(r \bar \v) \|_{L^2_{-1}(T)}
       \lesssim h_T^{-\epsilon} \lvert r_3 - r_2 \rvert \| \v \|_{L^2_{-1 + 2\epsilon}(T)},
    \end{align}
    what concludes the local estimate. 

    For the global estimate, we also need an estimate on triangles $T \in \mathcal{T} \setminus \TR$ away from the rotation axis. By quasi uniformity of the mesh we have on $T$ that
    \begin{align*}
    0 < r_\text{min} \le r \le r_\text{max} \leq r_\text{min} + h_T,
    \end{align*}
    where $r_\text{min} = \min_{x \in T} r(x)$ and $r_\text{max} = \max_{x \in T} r(x)$.
    Recalling the properties of the standard RT0 interpolation operator, this gives
    \begin{align*} 
      \Vert r \v - \Pi_{\mathrm{RT}0}^\mathrm{axi}(r \v) \Vert_{L^2_{-1}(T)} \le r_\text{min}^{-1/2} \Vert (1-\Pi_{\mathrm{RT}0})(r \v) \Vert_{L^2(T)} \le h_T r_\text{min}^{-1/2} \Vert \nabla (r \v) \Vert_{L^2(T)}. 
    \end{align*}
    The chain rule and $\| \cdot \|_{L^2} \leq r_{\max}^{1/2 - \epsilon} \| \cdot \|_{L^2_{-1 + 2\epsilon}}$ gives
    \begin{align*}
       h_T r_{\text{min}}^{-1/2} \Vert \nabla (r \v) \Vert_{L^2(T)} \le h_T r_{\text{min}}^{-1/2} r_{\text{max}}^{1/2 - \epsilon} \left( \| \nabla \v \|_{L^2_{1+2\epsilon}(T)} + \Vert \v \Vert_{L^2_{-1+2\epsilon}(T)}\right).
    \end{align*}
    Since $r_{\text{max}}/r_{\text{min}} \simeq 1$, we have that $h_T r_{\text{min}}^{-1/2} r_{\text{max}}^{1/2 - \epsilon} \simeq h_T^{1-\epsilon}$. 
    Summing over all triangles $T \in \mathcal{T} \setminus \TR$  and using the embedding from Remark~\ref{rem:embedding_mr82} results in
    \begin{align*}
      \Vert h_\mathcal{T}^{-1 + \epsilon} (r \v - \Pi_{\mathrm{RT}0}^\mathrm{axi}(r \v)) \Vert_{L^2_{-1}(\Omega)} \lesssim \Vert \nabla \v \Vert_{L^2_1(\Omega)} + \Vert \v \Vert_{L^2_{-1+2\epsilon}(\Omega)} \lesssim \| \v \|_{H^1_1(\Omega)} \lesssim \Vert \v \Vert_{\V},
    \end{align*}
    which concludes the proof. 

    For $\Pi_{\mathrm{BDM}1}^\mathrm{axi}$, by Lemma~\ref{lem::consistency}, constants are exactly interpolated on all elements, therefore the last term in \eqref{eq::interpolationestimate_split} also vanishes for type 1 triangles. Furthermore, the proof can be mimicked when considering $T \in \mathcal{T} \setminus \TR$. For this, instead of the weighted Poincar\'e inequality from Lemma \ref{lem:weighted_poincare}, a similar inequality also holds 
    for functions $\v \in \bm{H}^1_1(T)$ with $\int_T v \mathit{dx} = 0$ (where $T \in \mathcal{T} \setminus \TR$). This follows immediately by the standard Poincar\'e inequality in $L^2$ and
    \begin{align*}
       \| \v \|_{L^2_1(T)} \leq r_{\max}^{1/2} \| \v \|_{L^2(T)}
       \lesssim r_{\max}^{1/2} h_T \| \nabla \v \|_{L^2(T)}
       \lesssim \left(\frac{r_{\max}}{r_{\min}}\right)^{1/2} h_T \| \nabla \v \|_{L^2_1(T)}.
    \end{align*}
    Thus, we obtain the claim with $\epsilon = 0$.
\end{proof}

Finally, the modified interpolation estimate allows us to show the following improved consistency error estimate for $ \Pi_{\mathrm{RT}0}^\mathrm{axi}$ (and analogously for $\Pi_{\mathrm{BDM}1}^\mathrm{axi}$).

\begin{theorem}\label{thm:recons_consistency_error_modified}
If $\bf \in L^2_{1}(\Omega)$ (and hence $\mathbb{P} (\Delta_\text{axi} \u) \in L^2_1(\Omega)$), the consistency error \eqref{eqn:consistency_error}
for $\Pi \in \lbrace \Pi_{\mathrm{RT}0}^\mathrm{axi}, \Pi_{\mathrm{BDM}1}^\mathrm{axi}\rbrace$ can be estimated by
  \begin{align*}
    \| \mathbb{P} (\Delta_\text{axi} \u) \circ (\mathbb{I} - \Pi) \|_{\V_{h,0}^\star}
    \lesssim \| h_\mathcal{T}^{1-\epsilon} \mathbb{P} (\Delta_\text{axi} \u) \|_{L^2_1},
  \end{align*}
  where $\epsilon$ can be chosen according to Theorem~\ref{thm:modified_interpolation_estimate}.
\end{theorem}
\begin{proof}
    This follows directly by applying the Cauchy-Schwarz inequality and Theorem~\ref{thm:modified_interpolation_estimate}, i.e.,
    \begin{align*}
(\mathbb{P} (\Delta_\text{axi} \u), r \w_h - \Pi (r \w_h))_{L^2}
  & \leq \| h_\mathcal{T}^{1-\epsilon} \mathbb{P} (\Delta_\text{axi} \u) \|_{L^2_1} \| h_\mathcal{T}^{-1+\epsilon} (1-\Pi) (r \w_h) \|_{L^2_{-1}}\\
  & \lesssim \| h_\mathcal{T}^{1-\epsilon} \mathbb{P} (\Delta_\text{axi} \u) \|_{L^2_1} \|\w_h \|_{\V}.
  \end{align*}
\end{proof}

\section{Numerical examples}
\label{sec:numerical_examples}

This section compares a classical Bernardi--Raugel (BR) finite element
discretization with a pressure-robustly modified discretization
with $\Pi \in \lbrace \Pi_{\mathrm{RT}0}^\mathrm{axi}, \Pi_{\mathrm{BDM}1}^\mathrm{axi}\rbrace$.
If not mentioned otherwise the right-hand side functionals and boundary data integrals are evaluated with a quadrature rule of order 10,
while the bilinear form $a$ is evaluated with a quadrature rule of order 4.
The used grids are unstructured triangular grids without any refinement
close to the rotation axis.

\subsection{Example 1: linear stagnation flow} 

As a first example, we consider the axisymmetric Stokes problem with the exact solutions and data
\begin{align*}
    p_{\text{ex}} = r^{7/4} + z^2,
    \quad \u_{\text{ex}} & = [r, -2z]^T, \quad \bf = -\nu \Delta_\text{axi} \u_{\text{ex}} + \nabla p_{\text{ex}}.
\end{align*}
Observe, that $\Delta_\text{axi} \u_{\text{ex}} = \bm{0}$
and $\nabla p_{\text{ex}} \in L^2(\Omega)$, and in particular, the exact velocity solution is in the discrete ansatz space of the \BR finite element method.

\begin{figure}[!htbp]
  \centering
  \includegraphics[width=\textwidth]{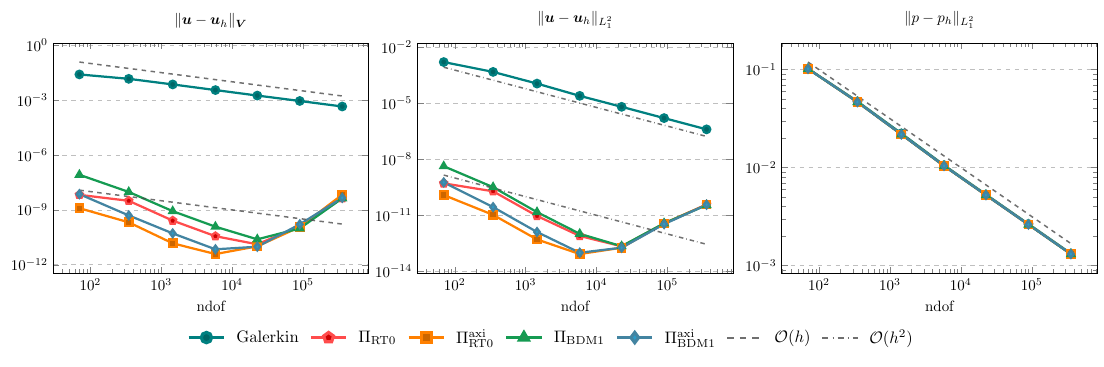}
  \caption{\label{fig:example1}Error convergence histories in Example 1 for $\nu = 1$.}
\end{figure}

Figure~\ref{fig:example1} shows that indeed all variants that use
one of the suggested reconstruction operators, have errors close to zero,
while the classical \BR finite element method
without the reconstruction operator shows errors in the range
of the pressure best-approximation error. Note that this error scales
with $\nu^{-1}$, which is demonstrated in the next examples.

\subsection{Example 2: smooth example}
\label{sec:example2}

This example considers smooth data, where we expect that all methods to show optimal convergence rates, namely
\begin{align*}
    p_{\text{ex}} = \sin (\pi(r^2+z^2)),
    \quad \u_{\text{ex}} =
    \begin{pmatrix}
      r^3 \sin z\\
      4r^2 \cos z
    \end{pmatrix}, \quad \bf = -\nu \Delta_\text{axi} \u_{\text{ex}} + \nabla p_{\text{ex}}.
\end{align*}

\begin{figure}[!htbp]
  \centering
  \includegraphics[width=0.9\textwidth]{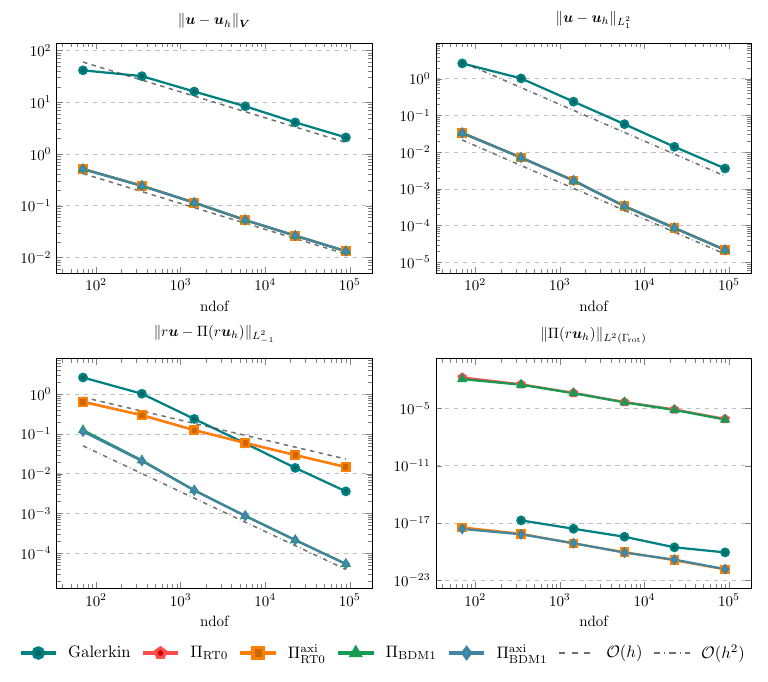}
  \caption{\label{fig:example2}Error convergence histories in Example 2 for $\nu = 10^{-3}$.}
\end{figure} 

\begin{figure}[!htbp]
  \centering
  \includegraphics[width=0.97\textwidth]{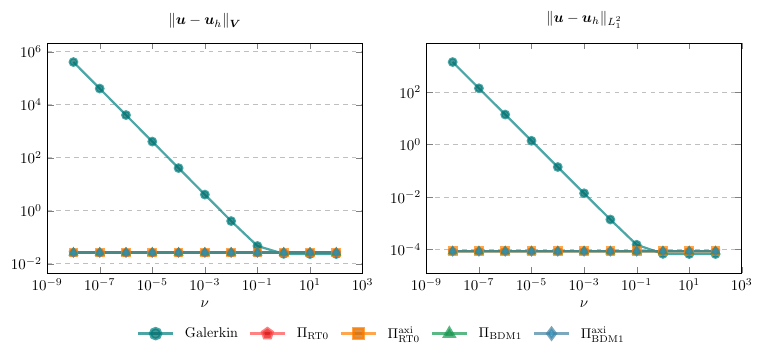}
  \caption{\label{fig:example2_variable_nu}Error convergence histories in Example 2 with respect to variable $\nu$ on a fixed mesh with about $22.000$ degrees of freedom.}
\end{figure}

Figure~\ref{fig:example2} shows the convergence history for several
norms of interest for $\nu = 10^{-3}$. In the energy and $L^2_1$
error of $\u - \u_h$ all variants of the reconstruction operators perform equally well and the classical \BR finite element method shows larger errors by about two orders of magnitude.
As expected, the $L^2_{-1}$ error of $r\u - \Pi(r\u_h)$ converges only linearly for $\Pi_{\mathrm{RT}0}$ and $\Pi_{\mathrm{RT}0}^\mathrm{axi}$, while the $\mathrm{BDM}1$ variants show optimal quadratic convergence and much smaller errors than the Galerkin method. 
This is of relevance when the reconstruction operator is used as a post-processing in coupled transport \cite{FMR26}.
The last subplot for the norm $\| \Pi(r\u_h) \|_{L^2(\GammaR)}$ confirms that the modified reconstructed functions really vanish at the rotation axis. However, at least in this smooth example, this property seems to have no qualitative impact on the convergence of the other norms.

Figure~\ref{fig:example2_variable_nu} shows errors on a fixed mesh but for
different values of $\nu$. The plots confirm the locking behavior of the classical method
in the sense that at some point the errors scale like $\nu^{-1}$. All modified methods show no locking behavior even for very small viscosities $\nu$.

\subsection{Example 3: data in $L^2_1(\Omega) \setminus L^2(\Omega)$}

\begin{figure}[!htbp]
  \centering
  \includegraphics[width=0.9\textwidth]{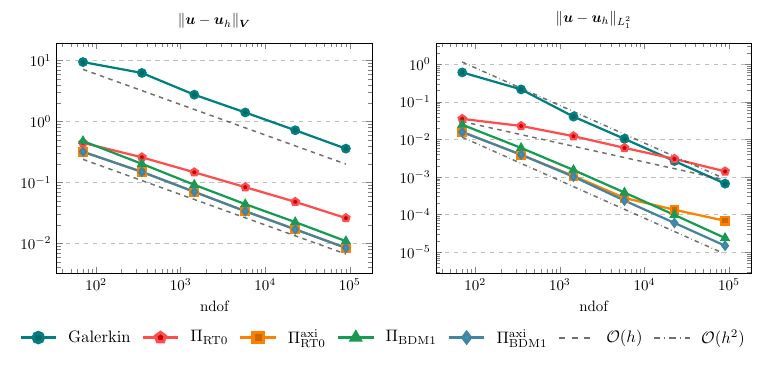}
  \caption{\label{fig:example3Conv}Error convergence histories in Example 3 for $\nu = 10^{-3}$.}
\end{figure}

This example studies the axisymmetric Stokes problem for
\begin{align*}
  p_{\text{ex}} = r^{1/2} - \frac{8}{9},
  \quad \u_{\text{ex}} =
  \begin{pmatrix}
    r^{2.1}\\
    -3.1 r^{1.1} z
  \end{pmatrix}
  , \quad \bf = -\nu \Delta_\text{axi} \u_{\text{ex}} + \nabla p_{\text{ex}}.
\end{align*}
Here, it holds $\u_{\text{ex}} \in \V$ and $p \in L^2_1(\Omega)$, but $\bf \in L^2_1(\Omega) \setminus L^2(\Omega)$.
In other words, the assumptions for the consistency error estimates in Theorem~\ref{thm:recons_consistency_error}
are not satisfied, while the ones in Theorem~\ref{thm:recons_consistency_error_modified} are satisfied.
In Figure~\ref{fig:example3Conv}, we compare the convergence rates for the fixed parameter $\nu=10^{-3}$ in the velocity norms.
This time, the errors with the classical reconstruction
$\Pi_{\mathrm{RT}0}$ are much larger than of the other
reconstruction operators. Also with
$\Pi_{\mathrm{BDM}1}$ the energy errors are not as small as
with the modified reconstruction operators $\Pi_{\mathrm{RT}0}^\mathrm{axi}$ and $\Pi_{\mathrm{BDM}1}^\mathrm{axi}$. Interestingly,
the $L^2_1$ velocity error converges suboptimally for
$\Pi_{\mathrm{RT}0}^\mathrm{axi}$. This was not the case in the smooth example.

\begin{figure}[!htbp]
  \centering
  \includegraphics[width=0.9\textwidth]{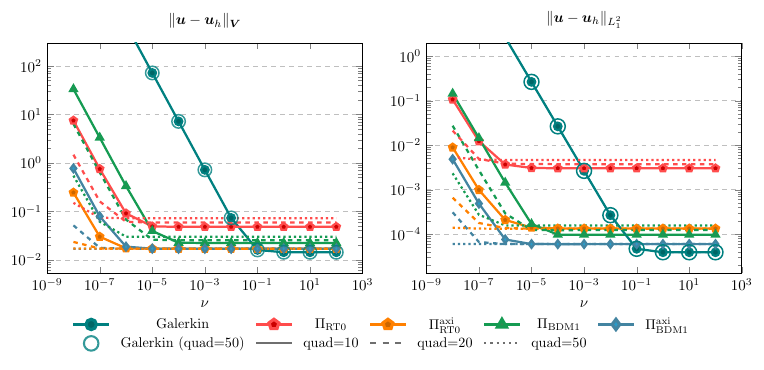}
  \caption{\label{fig:example3}Error convergence histories in Example 3 with respect to variable $\nu$ on a fixed mesh for quadrature rules of increasing orders $10$, $20$, and $50$.}
\end{figure}

Figure~\ref{fig:example3} shows the dependency of the error norms
on the viscosity $\nu$ on a fixed mesh. While the results are very
similar for the classical method, there are some interesting differences
among the methods with reconstruction operator. First, all methods
show an increase of the error that scales with $\nu^{-1}$ below some
critical value for $\nu$. This is related to the quadrature error
in the right-hand side and an expected behavior that does not
contradict pressure robustness, since Theorem~\ref{thm:apriori_error_estimate} assumes exact evaluation of all integrals. 
In the previous example the data was very smooth, so
the quadrature error was small. Due to the lower regularity of the data in this example the quadrature error is much more important.

However, it is interesting that the critical value for the modified reconstruction operators is much smaller than that for the classical reconstruction operators for the same quadrature rule.
Another interesting observation is that the error for the unmodified reconstruction operators increases when the quadrature order is increased and quadrature points move closer to the rotation axis. This behavior seems to be related to the missing vanishing-on-axis-property of the unmodified reconstruction operators and the missing $L^2_1$-regularity of the data.
The modified method and the classical method do not show an increase of the error when the quadrature order is increased.
More importantly, it can be observed that the classical reconstruction
operator $\Pi_{\mathrm{RT}0}$ yields much larger errors than the
Galerkin method for moderate $\nu$.

\bibliographystyle{siamplain}
\bibliography{main}

\end{document}